\documentclass[12pt]{article}
\newcommand{\blind}{1}

\usepackage{todonotes}

\usepackage{paralist, amsmath, amsthm, amssymb, color, graphicx, hyperref}
\DeclareGraphicsExtensions{.jpg,.pdf,.eps}
\usepackage{amssymb}
\usepackage{setspace}
\usepackage{mathtools}
\usepackage{color}
\usepackage{cancel}
\usepackage{algorithm,algorithmic}

\usepackage{verbatim}
\usepackage[round]{natbib}

\usepackage{stmaryrd}

\usepackage{xfrac}

\usepackage{mathrsfs}
\usepackage{amscd}
\usepackage{dsfont}
\usepackage{tikz}
\usepackage{thm-restate}

\usepackage{enumerate}
\usepackage{xr}

\theoremstyle{plain}
\newtheorem{thm}{Theorem}[section] 
\newtheorem{lem}[thm]{Lemma}

\newtheorem{cor}[thm]{Corollary}

\theoremstyle{definition}
\newtheorem{defn}[thm]{Definition} 

\newtheorem{example}[thm]{Example}
\newtheorem{remark}[thm]{Remark}

\usepackage{listings}
\usepackage{color} 
\definecolor{mygreen}{RGB}{28,172,0} 
\definecolor{mylilas}{RGB}{170,55,241}
\definecolor{mygray}{gray}{0.95}

\newcounter{algsubstate}

\newcommand{\mscr}[1]{\mathscr{#1}}

\newcommand{\twid}[1]{\widetilde{#1}}

\newcommand{\ZZ}{\mathbb{Z}}
\newcommand{\RR}{\mathbb{R}}
\newcommand{\NN}{\mathbb{N}}

\newcommand{\EE}{\mathbb{E}}

\newcommand{\de}{\delta}
\newcommand{\ep}{\epsilon}
\newcommand{\sdp}{s}
\newcommand{\ndp}{n_{\mathrm{dp}}}

\usepackage{mathtools}

\newcommand{\defeq}{\vcentcolon=}

\DeclareMathOperator{\var}{Var}

\newcommand{\ol}{\overline}

\allowdisplaybreaks

\date{}
\begin{document}

\def\spacingset#1{\renewcommand{\baselinestretch}%
{#1}\small\normalsize} \spacingset{1}


\if1\blind
{
  \title{\bf Statistical Inference for Privatized Data with Unknown Sample Size}
  \author{Jordan Awan\thanks{
    Jordan Awan gratefully acknowledges the support of NSF awards  SES-2150615 and SES-2610910.}\hspace{.2cm}\\
    Department of Statistics, University of Pittsburgh\\\\
    Andr\'es F. Barrientos\thanks{The research of Andr\'es F. Barrientos was supported by the National Science Foundation National Center for Science and Engineering Statistics [49100422C0008]}\hspace{.2cm} \\
    Department of Statistics, Florida State University\\\\
    Nianqiao Ju\hspace{.2cm} \\
    Department of Mathematics, Dartmouth College}
  \maketitle
} \fi

\if0\blind
{
  \bigskip
  \bigskip
  \bigskip
  \begin{center}
    {\LARGE\bf Statistical Inference for Privatized Data with
    \smallskip 
    
    Unknown Sample Size}
\end{center}
  \medskip
} \fi

\bigskip
\begin{abstract}
We develop both theory and algorithms to analyze privatized data in \emph{unbounded differential privacy} (DP), where even the sample size is considered a sensitive quantity that requires privacy protection. 
   We show that the distance between the sampling distributions under unbounded DP and bounded  DP goes to zero as the sample size $n$ goes to infinity, provided that the noise used to privatize $n$ is at an appropriate rate; we also establish that Approximate Bayesian Computation (ABC)-type posterior distributions converge under similar assumptions. 
   We further give asymptotic results in  regimes where the privacy budgets vary, establishing similarity of sampling distributions as well as showing that the MLE in the unbounded setting converges to the bounded-DP MLE. 
  To facilitate valid, finite-sample Bayesian inference on privatized data under unbounded DP, we propose a reversible jump MCMC algorithm which extends the data augmentation MCMC of \citet{ju2022data}. 
   We also propose a Monte Carlo EM algorithm to compute the MLE from privatized data in both bounded and unbounded DP. 
   We apply our methodology to analyze a linear regression model as well as a 2019 American Time Use Survey Microdata File which we model using a Dirichlet distribution.
   \end{abstract}

\noindent%
\vfill

\section{Introduction}
Differential privacy (DP) \citep{dwork2006calibrating} is the leading framework in privacy protection. A randomized algorithm/mechanism satisfies DP if it ensures that for two ``neighboring datasets,'' the output distributions are ``similar.'' Variants of DP formalize the concepts of ``neighboring'' and ``similar'' differently. Similarity is often measured in terms of either a divergence, or bounds on the probability of obtaining any particular output. The two most common choices for ``neighboring datasets'' are 1) \emph{bounded DP}, where two datasets are neighbors if they have the same sample size and differ in one entry and 2) \emph{unbounded DP}, where two datasets are neighbors if they differ by the addition or removal of one entry; in this case the sample sizes of the two datasets differ by 1\footnote{ The terms bounded and unbounded DP were introduced in \citet{kifer2011no} and further used in the text \citet{li2017differential}. Other literature uses terms such as change-DP and add-delete-DP.}. The key difference between these two frameworks is that in unbounded DP, the sample size is considered to be a sensitive quantity, requiring privacy protection, which makes statistical inference more challenging. The sample size may need protection, especially when membership in the dataset is based on a sensitive quantity (e.g., dataset of patients with the sickle cell trait); this is because the sample size amounts to a count statistic, which is known to be sensitive to attacks (see ``Queries Over Large Sets are Not Protective'' \citealp[p7-p8]{Dwork2014}).

 While there are many DP mechanisms designed for either bounded DP or unbounded DP, several implementations of DP systems have focused on the unbounded setting \citep{mcsherry2009privacy,wilson2020differentially,rogers2021linkedin,amin2022plume}, likely due to its stronger privacy guarantee. On the other hand, for DP statistical inference, previous literature has focused on the bounded DP cases, as this allows the sample size $n$ to be available. 
For example, \citet{bernstein2018differentially,bernstein2019differentially,kulkarni2021differentially,ju2022data} all develop MCMC samplers to simulate from either an exact or approximate posterior distribution given the privatized data, and all of these methods work in bounded DP. 
In the frequentist setting, \citet{wang2018statistical} developed an asymptotic regime to derive statistical approximating distributions under bounded DP. There have been several works on private confidence intervals 
\citep{karwa2018finite,wang2019differentially,drechsler2022nonparametric,covington2025unbiased,awan2025simulation}, 
and hypothesis tests 
\citep{gaboardi2016differentially,awan2018differentially,barrientos2019differentially,awan2020differentially}, 
all of which use bounded DP. \citet{ferrando2022parametric} propose using the parametric bootstrap for inference on privatized data and 
\citet{brawner2018bootstrap,wang2025differentially} developed a non-parametric bootstrap for private inference, also in the bounded DP setting. 

   This paper is the first to investigate how to perform valid statistical inference in unbounded DP, with unknown $n$.  Our main theoretical questions are: How do sampling distributions differ when conditioning on the true value of $n$ versus a privatized $\ndp$? Similarly, when is it asymptotically justified to plug in the observed value of $\ndp$ for the actual sample size $n$?  
Since asymptotic methods are not always applicable in finite-sample settings, we also develop numerical algorithms to directly target posterior inference and MLE estimation in the unbounded DP setting. Our contributions are as follows:


\begin{itemize}
  \setlength\itemsep{0em}
    \item We derive bounds for the privacy budget used to privatize $n$, under which we show consistency of the sampling distribution for privatized data with unknown $n$ relative to the sampling distribution for privatized data with known $n$ as  $n\rightarrow\infty$. We also show that approximate Bayesian computation (ABC)-type posterior distributions under bounded and unbounded DP approach each other under similar assumptions. 
    \item In the asymptotic regime where the privacy budgets for $n$ and other summaries vary, we give conditions under which the unbounded DP sampling distribution approaches the bounded DP sampling distribution. In particular, we establish conditions for the MLE in unbounded DP to converge to the MLE for bounded DP.
        \item We extend the Metropolis-within-Gibbs algorithm of \citet{ju2022data} to the unbounded setting, using a reversible jump MCMC methodology. Similar to \citet{ju2022data}, 
        we have lower bounds on the acceptance probabilities in the algorithm. We also establish  conditions for the Markov chain to be  ergodic or geometrically ergodic.
    \item We propose a Monte Carlo EM algorithm to compute the MLE in both bounded and unbounded DP, using the MCMC method to approximate the expectation step. 
    \item  We apply our methodology to linear regression models, which we study through simulations, and analyze a 2019 American Time Use Survey Microdata File.
\end{itemize}

\noindent {\bf Organization: } In Section~\ref{s:DP} we review background on DP, 
establish our notation, and clarify the problems of interest. 
In Sections~\ref{s:sampleInf} and \ref{s:epInf}, we develop asymptotic results as either the sample size or the privacy parameters vary. 
We develop computational methods in Section~\ref{s:computational}, including  a reversible jump MCMC algorithm for Bayesian inference under unbounded DP in Section~\ref{s:MCMC} and a Monte Carlo EM algorithm to approximate the MLE in Section~\ref{s:MCEM}. In Section~\ref{s:simulations} we conduct numerical experiments, 
and Section~\ref{s:discussion} concludes with a discussion. 
Proofs and technical details are in the Supplementary Materials.

\section{Preliminaries and problem setup}\label{s:prelim}
\subsection{Differential privacy}\label{s:DP}

Differential privacy, introduced by \citet{dwork2006calibrating}, is a probabilistic framework which quantifies the privacy protection of a given mechanism (randomized algorithm). DP mechanisms require the introduction of randomness into the procedure, which is designed to obscure the contribution of any one individual. To satisfy differential privacy, we require that when applied to two ``neighboring datasets,'' the resulting distributions of outputs from the mechanism are ``similar.''  
Given a space of input datasets, $\mscr X$, as well as a metric $d$ on $\mscr X$, two datasets $X,X'\in \mscr X$ are called \emph{neighboring}  if $d(X,X')\leq 1$. 

\begin{defn}
[Differential privacy: \citealp{dwork2006calibrating}]\label{def:DP}
Let $\ep\geq0$ and $\delta\in [0,1]$. Let $M:\mscr X\rightarrow \mscr Y$ be a mechanism and let $d:\mscr X\times \mscr X\rightarrow \RR^{\geq 0}$ be a metric. We say that $M$ satisfies $(\ep,\de)$-Differential Privacy ($(\ep,\de)$-DP) if $P(M(x)\in S)\leq \exp(\ep) P(M(x')\in S)+\delta,$ 
for all $x,x'\in \mscr X$, such that $d(x,x')\leq 1$ and all measurable sets of outputs $S\subset \mscr Y$. If $\de=0$, we simply write $\ep$-DP in place of $(\ep,0)$-DP.
\end{defn}

In Definition \ref{def:DP}, $\ep$ and $\de$ are called the privacy parameters, which govern how strong the privacy guarantee is. Smaller $\ep$ and $\de$ values give stronger privacy constraints. Typically, $\ep$ is a small constant often $\leq 1$, and $\de$ is usually chosen to be very small, such as $\de \ll 1/n$. 

Many databases can be expressed as an ordered sequence $x=(x_1,x_2,\ldots, x_n)$ for $x_i \in \mscr X_1$, where $x_i$ is the information for individual $i$. In this case, there are two main flavors of DP: \emph{bounded DP}, which takes $\mscr X = \mscr X_1^n$ and uses the Hamming metric for adjacency, and \emph{unbounded DP}, which takes $\mscr X = \mscr X_1^*=\bigcup_{n=1}^\infty \mscr X_1^n$ and considers $x$ and $x'$ to be neighbors if they differ by the addition or removal of one entry. If a mechanism satisfies $\ep$-(unbounded)DP, then it satisfies $2\ep$-(bounded)DP  \citep[Section 2.1.1]{li2017differential}. In bounded DP, the sample size is part of the privacy definition, and is implicitly a public quantity; in unbounded DP, the sample size is not publicly known and requires protection.


\subsection{Asymptotics}\label{s:asymp}
For two probability measures $P$ and $Q$ on a $\RR^d$, the total variation distance between $P$ and $Q$ is $\mathrm{TV}(P,Q) = \frac 12 \lVert P-Q\rVert=\sup_S |P(S)-Q(S)|$. A weaker metric is the KS distance: $\mathrm{KS}(P,Q) = \sup_R |P(R)-Q(R)|$, where the supremum is over axis-aligned rectangles.

We write $P_n\overset {KS}\sim Q_n$  if $\mathrm{KS}(P_n,Q_n)\rightarrow 0$ and  $P_n\overset {TV}\sim Q_n$ if $\mathrm{TV}(P_n,Q_n)\rightarrow 0$. If a distribution $P_n$ has an asymptotic limit (possibly with some rescaling), then if another distribution $Q_n\sim P_n$ in either sense, then it has the same asymptotic distribution. In cases where $P_n$ and $Q_n$ have respective densities $p_n(y)$ and $q_n(y)$ with respect to a common base measure, then with some abuse of notation, we may use $p_n$ and $q_n$ to represent $P_n$ and $Q_n$ when discussing KS or TV distances. 
For two positive numerical sequences $a_n$ and $b_n$, we write $a_n \sim b_n$ to mean that $a_n/b_n\rightarrow 1$.

\subsection{Problem setup}\label{s:setup}

The random variable $x$ represents the sensitive dataset,  which is commonly a matrix of dimension $n\times m$, and which is no longer available after privacy protection.  Note that the true sample size of $x$ is $n$, which is also no longer available after unbounded DP privacy protection. We assume that the privacy mechanism results in two outputs $s\in \RR^d$ and $\ndp\in \RR$, where $s$ is any DP summary  (possibly, but not necessarily based off of a sufficient statistic) and $\ndp$ is a DP estimate of $n$, commonly achieved by applying an additive noise mechanism to $n$.  We assume that $s$ and $\ndp$ are conditionally independent given $x$ and $n$; that is, neither one explicitly depends on the other. Finally, $\theta\in \RR^p$ is the parameter that generated $x$, which is a random variable with a prior distribution. We also assume a prior distribution on the unobserved $n$. Since we will consider both Bayesian and frequentist problems, for frequentist settings, we condition on $\theta$ and disregard the prior distribution. 

\begin{remark}
     In our notation, we separate the general DP summary $s$ from $\ndp$, which is a DP estimate of $n$.  While some mechanisms may have $s$ depending on $\ndp$, our framework requires that these mechanisms do not depend on each other. For example, a private estimator of the mean could use $\ndp^{-1}\left(\sum_{i=1}^n x_i+L_1\right)$ where $L_1$ is an appropriately calibrated noise (such as Laplace), and $\ndp = n+L_2$ where $L_2$ is another noise variable; however, in our setting, we think of $s=\sum_{i=1}^n x_i+L_1$, which is the noisy \emph{sum} rather than the noisy \emph{average}. As a post-processing step, we can estimate the average with $s/{\ndp}$. There are also mechanisms that can directly estimate the mean without explicitly using $\ndp$, such as the KNG mechanism discussed in Example \ref{ex:LaplaceKNG}.
\end{remark}

Due to the large number of random variables in this problem, we will simply use $p$ to denote the pdf/pmf of a random variable; for example $p(\ndp|n)$ is the density of the privacy mechanism which produces $\ndp$ from a dataset with size $n$. 

The joint Bayesian model for the parameters and observed values is  
$p(\theta,s,n,\ndp) = \int p(s|x)p(x|\theta,n)p(\theta) p(\ndp|n) p(n) \ dx.$ 
The joint distribution of $s$ and $\theta$, given the true $n$, is 
$p(\theta,s|n) = p(s|\theta,n) p(\theta),$  where we have a ``prior'' for the sample size $n$ as well as for the parameter $\theta$. Note that we write $p(x|\theta,n)$ for the sampling distribution for the sensitive data; commonly given $n$ and $\theta$, $x$ consists of $n$ i.i.d. samples drawn from the model with parameter $\theta$.  Furthermore, $p(s|x,\theta,n)=p(s|x)$ since 1) the mechanism only operates on the data and does not use the parameter; so, given the data $x$, $s$ and $\theta$ are conditionally independent and 2) the sample size $n$ can be viewed as a function of the data $x$; thus, $p(s|x,n)=p(s|x)$. Similarly, the DP mechanism for $\ndp$ only depends on $n$, and so given $n$, $\ndp$ and $(x,\theta)$ are conditionally independent, which implies that $p(\ndp|n,x,\theta)=p(\ndp|n)$. 
The joint distribution given only $\ndp$ is 
\begin{equation}\label{eq:posterior}
\begin{split}
    p(\theta,s|\ndp) & = \frac{\sum_{k=1}^\infty p(s|\theta, n=k) p(\theta) p(\ndp|n=k) p(n=k)}{\sum_{k=1}^\infty p(\ndp|n=k) p(n=k)}\\
    &=\sum_{k=1}^\infty p(\theta,s|n=k) p(n=k|\ndp),
    \end{split}
\end{equation}
which is a mixture of the distributions $p(\theta,s|n)$, weighted by the posterior $p(n|\ndp)$. 

In terms of frequentist models, it is natural to consider $n$ to be part of the data, rather than part of the parameter. As such, rather than a ``prior'' distribution for $n$, we would have a model $p(n|\lambda)$, where $\lambda$ are the parameters that govern the distribution of $n$. In this case, we can also write the frequentist distribution for $s$ given $\ndp$:
\begin{equation}\label{eq:freq}
\begin{split}
    p(s|\theta,\lambda,\ndp) &= \frac{\sum_{k=1}^\infty p(s|\theta,n=k)p(\ndp|n=k)p(n=k|\lambda)}{\sum_{k=1}^\infty p(\ndp | n=k) p(n=k|\lambda)}\\
    &=\sum_{k=1}^\infty p(s|\theta,n=k) p(n=k|\ndp,\lambda).
    \end{split}
\end{equation}

In the remaining paper, we will prefer to write $p(n)$ rather than $p(n|\lambda)$ for simplicity, unless details of the specific distribution of $n$ are needed. 

 Note that both \eqref{eq:posterior} and \eqref{eq:freq} do not involve the original data $x$ or $n$ as these quantities are not available after the privatization. Both quantities have \emph{marginalized out} the data $x$ along with the sample size $n$, which can be viewed as missing data or latent variables. This approach has been previously considered in the literature (e.g., \citealp{williams2010probabilistic,gong2022exact,ju2022data,chen2025particle}).

{\noindent \bf Problem of Interest: }  The main theoretical question of interest, which is tackled in Sections \ref{s:sampleInf} and \ref{s:epInf}, is as follows: how do sampling distributions differ when given the true value of $n$ versus a privatized $\ndp$? 
To make our theory precise, we use $n_0$ as a generic value which could represent the observed value of either $\ndp$ (unbounded DP) or $n$ (bounded DP) and compare distributions where one conditions on $\ndp=n_0$ and the other conditions on $n=n_0$.  We also consider a more traditional frequentist regime where we condition on $n=n_0$, but plug in $\ndp$ in place of $n$. Section \ref{s:sampleInf} establishes conditions that ensure that such pairs of distributions asymptotically coincide as $n_0\rightarrow \infty$, and Section \ref{s:epInf} studies the case where the privacy budgets for $\ndp$ and $s$ jointly vary in specific ways.

\section{Asymptotic results as sample size increases}\label{s:sampleInf}
We consider the relation between distributions under bounded and unbounded DP, and establish when these distributions have the same asymptotic properties. We also discuss potential applications of these results, such as justifying approximate plug-in procedures. 

\subsection{Convergence of sampling and joint distribution}\label{s:likelihood}
In this section, we establish conditions for distributions under unbounded and bounded DP to converge in KS or TV distance. The key conditions are specified in A1 and A2 below. Intuitively, A1 says that in bounded DP the random variable $s$ has a Central Limit Theorem (CLT)-type asymptotic distribution, where we allow an arbitrary rate of convergence to a general mean-zero continuous distribution. A2 requires that the difference between $\ndp$ and $n$ is not too large. Interestingly, we find that the threshold for ``too large'' depends on the nature of the asymptotic distribution in A1.

\begin{itemize}
  \setlength\itemsep{0em}
    \item A1:  The sequence of random variables, $(s|\theta,n=n_0)$, indexed by $n_0$, has an asymptotic distribution: there exists $b\in (0,1]$, $a\in \RR$, and a function $g:\Theta\rightarrow \RR^d$ such that    
    $n_0^b(n_0^{-a} s-g(\theta))\overset d \rightarrow X,$ 
     as $n_0 \rightarrow \infty$, where $X$ is a mean-zero continuous random variable. (e.g., Central Limit Theorem: $b=1/2$, $a=1$ or $a=0$, and $X\overset d =N(0,I)$)
    \item A1': A1 holds, but where convergence is in total variation.
    \item A2: 
     Assuming A1 holds, let $a_{n_0}$ be a sequence such that 
        if $a=0$ in A1, then $a_{n_0} = o(n_0)$, and 
        if $a\neq 0$ in A1, then $a_{n_0} = o(n_0^{1-b})$. Either 
            (a) assume that $\ndp$ takes values in $\ZZ$ and $(\ndp-n_0|n=n_0)=O_p(a_{n_0})$, or 
            (b) assume that for all $n_0\in \ZZ^+$, $(n-n_0|\ndp=n_0)=O_p(a_{n_0})$.
\end{itemize}

While A1 and A2(a) are relatively intuitive, and can be easily verified for many DP summary statistics such as the Laplace and Gaussian mechanisms, A2(b) is less straightforward to verify, as it concerns itself with a posterior distribution.  In Lemmas A.1 and A.2 in the Supplementary Materials, we show that when using a flat prior for $n$, the (continuous) Laplace mechanism for $\ndp$ with privacy parameter $\epsilon$ satisfies $(n-n_0|\ndp=n_0)=O_p(1/\ep)$, and when using any additive noise mechanism for $\ndp$ with a symmetric and integer-valued noise, the rate of $(n-n_0|\ndp=n_0)$ is the same as $(\ndp-n_0|n=n_0)$. For example, both the discrete Laplace and discrete Gaussian mechanisms are symmetric and integer-valued. Thus,  Assumptions A1 and A2 are satisfied by some common additive noise mechanisms.

The alternative assumption A1' ensures that our convergence results hold in total variation, rather than only in distribution. Convergence in total variation is a stronger notion of convergence which, due to the data processing inequality, ensures that all transformations (such as statistics/estimators) of the random variables also converge. If $s$ is an additive privacy mechanism applied to a sample mean, then under a number of common settings, A1' holds; see Lemma A.6 in the Supplementary Materials for examples. 



\begin{restatable}{thm}{thmconvergence}\label{thm:convergence}
 1. Under A1 and A2(a), we have that $p(n^b(n^{-a}s-g(\theta))|\theta,n)\overset {KS}\sim p(\ndp^b(\ndp^{-a}s-g(\theta))|\theta,n)$ as $n\rightarrow \infty$. If $p(\theta)$ is a proper prior, we also have $p(\theta, n^b(n^{-a}s-g(\theta))|n)\overset {KS}\sim p(\theta, \ndp^b(\ndp^{-a}s-g(\theta))|n)$.

2. Under A1 and A2(b), we have that $p(s|\theta,n=n_0)\overset {KS}\sim p(s|\theta,\ndp=n_0)$ as $n_0\rightarrow \infty$. If $p(\theta)$ is a proper prior, then $p(\theta,s|n=n_0)\overset {KS}\sim p(\theta,s|\ndp=n_0)$ as well. 

Replacing A1 with A1', all convergence results hold in total variation distance. 
\end{restatable}

Part 1 of Theorem \ref{thm:convergence} shows that the sampling distribution of $n^b(n^{-a}s-g(\theta))$ is asymptotically the same when $\ndp$ is plugged in for $n$. This could be used to justify a plug-in frequentist approximation. 
Part 2 addresses a slightly different idea: whether we observe the true $n=n_0$ or the privatized $\ndp=n_0$ (for large $n_0$), we would expect $p(s,\theta|n=n_0)$ and $p(s,\theta|\ndp=n_0)$ to be similar. This could also be interpreted as observing $\ndp=n_0$, but instead conditioning on $n=n_0$: Part 2 of Theorem \ref{thm:convergence} says that after conditioning on either variable, the distributions for $s$ (or $(s,\theta)$) are approximately the same, meaning that the plug-in conditional distribution is a suitable approximation to reality. 
That said, part 2 does not offer as clear an application in terms of approximate inference procedures: while the conditional distributions are similar, this does not necessarily imply that specific inference methods give similar results.  
%



The strength of convergence in total variation is that it implies that any estimator which depends on $s$ and $n_0$ also has the same asymptotic distribution under either $p(s|\theta,n=n_0)$ or $p(s|\theta,\ndp=n_0)$, whereas under KS convergence we cannot generally make this claim. See Corollary A.10 in the Supplementary Materials for a formal statement.

Based on Theorem \ref{thm:convergence} and how it depends on assumptions A1 and A2, we investigate some DP examples where one could design $s$ to have either $a=0$ or $a=1$ in A1.

\begin{example}\label{ex:LaplaceKNG}
We consider two privacy mechanisms, which result in either $a=0$ or $a=1$ in assumption A1. Detailed derivations can be found in Example \ref{ex:LaplaceKNGDetails} in the Supplementary Materials. We assume that $x_i$ are i.i.d., with support $\{0,1\}$, and we are interested in privately estimating the sample mean subject to $\epsilon$-DP. 

{\bf Laplace Mechanism:} One of the most common privacy mechanisms, the Laplace mechanism, produces $s = \sum_{i=1}^{n} x_i+L$, where $L\sim \mathrm{Laplace}(0,1/\ep)$ and $s$ satisfies $\ep$-DP. We see that $a=1$ and $b=1/2$. So, for A2(b), we require $n|(\ndp=n_0)=n_0+o_p(\sqrt {n_0})$.

{\bf $K$-Norm Gradient (KNG) Mechanism \citep{reimherr2019kng}:} 
The empirical risk for mean estimation is $\sum_{i=1}^n (x_i-s)^2$, with gradient $\sum_{i=1}^n -2(x_i-s)$. If $s\in [0,1]$,  the sensitivity of the gradient is $2$. KNG samples $s$ from the density proportional to 
\begin{align*}
    \exp\left(-(\ep/4)\left|2\sum_{i=1}^n(x_i-s)\right|\right)I(s\in [0,1])
    &=\exp\left(-(\ep n/2)\left | \ol x-s\right|\right)I(s\in [0,1]),
\end{align*}
which we identify as $s|\ol x\sim \mathrm{Laplace}_{[0,1]}(\ol x,2/(n\ep))$, the truncated Laplace distribution on $[0,1]$. By \citet{reimherr2019kng}, this mechanism satisfies $\ep$-DP. Here we have $a=0$ and $b=1/2$. So, to satisfy A2(b), we only require $n|(\ndp=n_0)=n_0+o_p(n_0)$.
\end{example}

\subsection{ABC posterior convergence}\label{s:abc}
While the previous section developed results for the sampling distribution, a Bayesian analyst may be more interested in whether the posterior distributions $p(\theta|s,n=n_0)$ and $p(\theta|s,\ndp=n_0)$ are similar. While we were unable to get a direct result on these quantities, in Theorem \ref{thm:posteriorRectangle}, we show that certain ``ABC-type'' posterior distributions do converge. 

\begin{restatable}{thm}{thmposteriorRectangle}\label{thm:posteriorRectangle}
Assume that $p(\theta)$ is a proper prior distribution, and that there exists a sequence of rectangles $R_{n_0}$ such that $p_{n_0}\defeq P(s\in R_{n_0}|n=n_0)$ satisfies both $p_{n_0}>0$ and $\mathrm{KS}_{n_0}/p_{n_0}\rightarrow 0$, where $\mathrm{KS}_{n_0} = \mathrm{KS}(p(s,\theta|n=n_0),p(s,\theta|\ndp=n_0))$. Suppose further that there exists $b'>0$, a sequence $\theta_{n_0}$, and a random variable $Y$ such that $n_0^{b'}(\theta-\theta_{n_0})|(n=n_0,s\in R_{n_0})\overset d\rightarrow Y$, which has continuous cdf $F_Y$. Then $n_0^{b'}(\theta-\theta_{n_0})|(\ndp=n_0,s\in R_{n_0})\overset d\rightarrow Y$ as well. Equivalently, $p(n_0^{b'}(\theta-\theta_{n_0})|n=n_0,s\in R_{n_0})\overset {KS}\sim p(n_0^{b'}(\theta-\theta_{n_0})|\ndp=n_0,s\in R_{n_0}).$ 
\end{restatable}
We interpret the posterior distributions in Theorem \ref{thm:posteriorRectangle} as being similar to approximate Bayesian computing (ABC) posteriors, where we condition on the data lying in a certain set, rather than conditioning on the exact value (the rectangles would be most closely related to a weighted $\ell_\infty$ distance). Since the rectangles in Theorem \ref{thm:posteriorRectangle} can be arbitrarily small, this gives evidence that the exact posteriors are similar as well. Thus, Theorem \ref{thm:posteriorRectangle} offers evidence that whether we correctly condition on $n_{\ndp=n_0}$ or erroneously condition on $n=n_0$, the posterior distributions are similar for large $n$.

\begin{remark}
In Theorem \ref{thm:posteriorRectangle}, it is often the case that $b'$ is equal to $b$ defined in A1; see for example, \citet{frazier2018asymptotic}.  Another special case is to have $p_{n_0}$ constant and $\mathrm{KS}_{n_0}\rightarrow 0$ such as via Theorem \ref{thm:convergence}. For example, in Theorem \ref{thm:posteriorRectangle}, a  Bernstein-von Mises-style CLT could have that $\theta|(s\in R_{n_0},n=n_0)$ converges to a normal distribution in which case  $b=1/2$, $\theta_{n_0}=\theta_0$, and $Y$ is a normal distribution. {Furthermore, if $\sqrt {n_0} (s-s_0)|(n=n_0)\overset d\rightarrow Z$, where $Z$ is a continuous distribution with full support on $\RR^d$ (e.g., Gaussian)},  then we may take $R_{n_0} = R^0_{n_0}/\sqrt {n_0}+s_0$,  where $R_{n_0}^0$ is a non-trivial rectangle (meaning that it has positive Lebesgue measure). In this case,  $p_{n_0}\rightarrow P_{Z\sim N(0,\Sigma)}( Z\in R_{n_0}^0)>0$ and by assuming A1 and A2(b), we get that $\mathrm{KS}_{n_0}\rightarrow 0$, which ensures that $\mathrm{KS}_{n_0}/p_{n_0}\rightarrow 0$. 
\end{remark}


\section{Convergence as privacy budgets vary}\label{s:epInf}

We now consider a different asymptotic regime, where the true sample size is held fixed, but the privacy budgets for $s$ and $\ndp$ vary. Our results offer theoretical insights into exactly what properties are needed for the sampling distributions to converge in this regime.

\subsection{Sampling and joint distributions}\label{s:likelihoodEp}
We will establish conditions for  either $\mathrm{TV}(p(s,n|\theta),p(s,\ndp|\theta))\rightarrow 0$ or $\mathrm{TV}(p(s|\theta,n=n_0),p(s|\theta,\ndp=n_0))\rightarrow 0$, where the limit is in terms of the privacy budgets for $s$ and $\ndp$.  The first quantity is bounded by the probability that $\ndp\neq n$:

\begin{restatable}{prop}{propTVprivacy}\label{prop:TVprivacy}
We have $\mathrm{TV}(p(s,n|\theta),p(s,\ndp|\theta))\leq P(\ndp\neq n)$. If $\ndp$ takes values in $\ZZ$, Markov's inequality gives the upper bound $\EE_{n,\ndp}|\ndp-n|$.
\end{restatable}

Without additional structure, such as assumptions A1 and A2(a),   Proposition \ref{prop:TVprivacy} says that for $\mathrm{TV}(p(s,n|\theta),p(s,\ndp|\theta))\rightarrow 0$, we need $P(\ndp=n)\rightarrow 1$. This could be achieved with a discrete noise mechanism and having the privacy budget for $\ndp$ go to infinity. Thus, if $\ndp$ is sufficiently accurate, plugging in $\ndp$ for $n$ may be asymptotically justified.

 To study $\mathrm{TV}(p(s|\theta,n=n_0),p(s|\theta,\ndp=n_0))\rightarrow 0$, we express the privacy budget of $s$ in terms of  $(0,\delta)$-DP, which is implied by all the main-line DP definitions (see Lemma \ref{lem:TVDP}).  Note that we are not advocating for $(0,\delta)$-DP to be used as a privacy guarantee for reporting purposes, as it is usually viewed as too weak of a framework. Rather, if someone is using $\epsilon$-DP, $(\epsilon,\delta)$-DP, $\mu$-GDP, or R\'enyi-DP, then no matter which framework is used, we can calculate a valid $(0,\delta)$-DP guarantee, which we then use in our results. For example, if the $s$ mechanism satisfies $\epsilon$-DP, then we would substitute $\delta=\frac{\exp(\ep)-1}{\exp(\ep)+1}$ in Theorem \ref{thm:TVprivacy}. 

\begin{restatable}{thm}{thmTVprivacy}\label{thm:TVprivacy}
Let $p(s|x)$ be the distribution of a $(0,\delta)$-DP mechanism, let $p(x|\theta,n)$ be a distribution for $x$, where  $x$ consists of $n$ copies of i.i.d. variables, and let $p(n)$ be a (possibly improper) prior on $n$. Call $\gamma =  \EE_{n|\ndp=n_0} |n-n_0|$. Then 
$\mathrm{TV}(p(s|\theta,n=n_0),p(s|\theta,\ndp=n_0))
\leq \delta  \gamma,$ and  
 if $p(\theta)$ is a proper prior on $\theta$, then  
$\mathrm{TV}(p(s,\theta|n=n_0),p(s,\theta|\ndp=n_0))
\leq  \delta   \gamma.$
\end{restatable}


Lemma \ref{lem:Expectation} shows that for two common mechanisms for $p(\ndp|n)$, the discrete Laplace and discrete Gaussian mechanisms, and assuming a flat improper prior for $p(n)$, we can obtain explicit upper bounds on $\gamma$ in terms of the privacy parameter.

In order for the total variation distances in Theorem \ref{thm:TVprivacy} to go to zero,  we need the product $\delta \gamma\rightarrow 0$. This could be achieved by having $\delta\rightarrow 0$ or $\gamma\rightarrow 0$; the first asks the privacy guarantee on $p(s|x)$ to strengthen while the second asks for the privacy guarantee on $\ndp$ to weaken. We can also consider an interplay between the two: If the privacy guarantee for $s$ strengthens faster than the privacy guarantee for $\ndp$, it may also be possible for the product to go to zero, which we explore in the following example.





\begin{example}\label{ex:productEx}
    Suppose that $p(s|x)$ is an $\ep_s$-DP mechanism, $p(\ndp|n)$ is the discrete Laplace mechanism satisfying $\ep_n$-DP, and we use a flat prior for $p(n)$. Lemma \ref{lem:TVDP} shows that $\delta=\frac{\exp(\ep_s)-1}{\exp(\ep_s)+1}\sim \ep_s/2$, as  $\ep_s\rightarrow 0$ , and Lemma \ref{lem:Expectation} shows that $\gamma \leq \frac{2}{\exp(\ep_n)-1}\sim 2/\ep_n$ as $\ep_n\rightarrow 0$. In this case, as long as both $\ep_s,\ep_n\rightarrow 0$ and $\ep_s/\ep_n\rightarrow 0$, Theorem \ref{thm:TVprivacy} says that the distributions $p(s|\theta,n=n_0)$ and $p(s|\theta,\ndp=n_0)$ are close in total variation. This situation could arise when many DP statistics must be produced and, by the composition property of DP, the privacy budgets for each are limited \citep[Section 3.5]{Dwork2014}.
\end{example}

Similar to Theorem \ref{thm:posteriorRectangle}, we can analyze ABC-type posterior distributions as the privacy budgets vary. With total variation convergence, we are not limited to rectangles as in Theorem \ref{thm:posteriorRectangle}, but can instead condition on arbitrary sets with positive probability.

\begin{restatable}{thm}{thmABCprivacy}\label{thm:ABCprivacy}
 Let $\left(\ndp(\gamma)\right)_{\gamma>0}$ be a collection of privacy mechanisms such that for any $\gamma>0$,  $\EE_{n|\ndp(\gamma)=n_0}|n-n_0|=\gamma$. Assume the same setup as Theorem \ref{thm:TVprivacy}, where $p(s|x)$ satisfies $(0,\delta)$-DP. Then for any set $S$ such that $P(s\in S|n=n_0)>0$ and $P(s\in S|\ndp(\gamma)=n_0)>0$ for all $\gamma$ and $\delta$, we have that $p(\theta|s\in S,\ndp(\gamma)=n_0) \overset{TV}\sim p(\theta|s\in S,n=n_0)$, when $\delta\gamma \rightarrow 0$. 
\end{restatable}

Theorem \ref{thm:ABCprivacy} shows that approximate Bayesian posteriors also converge when the product of $\delta\gamma$ goes to zero. Thus, if the privacy budget for $p(s|x)$ is sufficiently smaller than the budget for $p(\ndp|n)$, this could justify plugging in $\ndp$ for $n$ in a Bayesian analysis. 

\subsection{Convergence of the MLE}\label{s:MLE}

In this section, we consider the maximum likelihood estimator (MLE) in the unbounded DP setting,  given the privatized values $s$ and $\ndp$, which we refer to as the unbounded-DP MLE.  We  show that as the privacy budget for $\ndp$ grows, this MLE converges to the MLE under bounded DP (under suitable regularity conditions). 

The DP MLE has been previously considered in the literature. \citet{karwa2016inference} develop an algorithm to compute the MLE for private exponential random graph models, \citet{gong2022exact} propose a Monte Carlo EM algorithm to compute the bounded DP MLE, and \citet{gaboardi2016differentially} use a two-step approximation of the MLE.

In bounded DP, the MLE of $\theta$ is $\hat{\theta}(n_0, \sdp) = \arg\max_\theta p(\sdp \mid n=n_0, \theta),$  where $p(s|n=n_0,\theta)$ is the marginal distribution of the private summary $s$, given the sample size $n=n_0$ and parameter $\theta$. 
In unbounded DP, however, the sample size is unknown. If we use a prior on $n$ that yields a proper posterior distribution $p(n \mid \ndp)$, then the MLE of $\theta$ is:
\begin{equation}
\label{eq:thetahat_ndp}
\hat{\theta}_{dp}(\ndp, \sdp) = \arg\max_{\theta} p(\sdp \mid \theta, \ndp)= \arg\max_{\theta} \sum_{k=1}^\infty p(n =k\mid \ndp)p(\sdp \mid n=k ,\theta).
\end{equation}

Here we show that the unbounded-DP MLE converges to the MLE we would get if the true sample size $n= n_0$ is known. The assumptions in Theorem \ref{thm:mleconsistency} are similar to those used in the Argmax Continuous Mapping Theorem \citep[Theorem 12.1]{sen2022gental}.

\begin{restatable}{thm}{thmmleconsistency}\label{thm:mleconsistency}
Fix sample size $n = n_0 > 0$. 
Suppose $\Theta\ni\theta$ is a metric space. 
Let the bounded DP likelihood be $L(\theta;s) \defeq p(s \mid \theta, n = n_0)$.
Write the unbounded DP likelihood as $L_{\epsilon}(\theta;\sdp,\ndp)\defeq \sum_{k=1}^{\infty} p(\sdp \mid \theta, n=k) p_{\epsilon}(n=k \mid \ndp)$, where $p_{\epsilon} (n \mid \ndp)$ is a proper posterior distribution   and $\epsilon$ is the privacy loss budget for the mechanism $p(\ndp\mid n)$. For each $\epsilon$, let $\hat{\theta}_{\epsilon}$ be a (potentially) random element of $\Theta$ such that 
    $
    L_{\epsilon}(\hat{\theta}_{\epsilon}; \ndp, \sdp) \ge \sup_{\theta \in \Theta} L_{\epsilon}(\theta; \ndp, \sdp) - o_{p}(1),$ where the $o_p(1)$ term is with respect to $\epsilon\rightarrow \infty$. Here $\hat{\theta}_{\epsilon}$ is the approximate argmax of the unbounded-DP likelihood. 
Suppose that the following conditions hold:
\begin{itemize}
  \setlength\itemsep{0em}
    \item Unique bounded-DP MLE: almost all sample paths $L(\theta; \sdp)$ are upper semi-continuous and each  has a unique maximum point $\hat{\theta}(n_0, \sdp)$.
    \item Uniform convergence on compact sets:  for every compact subset $K$ of $\Theta$, we have for all $\theta\in K$, $L_{\epsilon}(\theta; \sdp,\ndp) \overset{d}{\to} L(\theta; \sdp)$ in $\ell^{\infty}(K)$ as $\epsilon \to \infty$. The probabilities are over the randomness in both $\sdp$ and $\ndp$.
    \item Tightness condition: for every $c > 0$, there exists a compact set $K_c \subset \Theta$ such that 
    $\limsup_{\epsilon \to \infty} \mathbb{P}(\hat{\theta}_{\epsilon} \not\in K_c) \le c, \quad \mathbb{P}(\hat{\theta}(n_0, s) \not\in K_c) \le c.$ 
\end{itemize}
Then, we have $\hat{\theta}_{\epsilon} \overset d \rightarrow \hat{\theta}(n_0 , s)$ as $\epsilon \to \infty$.
\end{restatable}

 Example \ref{ex:mleconsistency} in the Supplementary Materials works out an example to illustrate when the conditions in Theorem \ref{thm:mleconsistency} can be satisfied. This example enforces some structure on the non-private summary, such as being a location-family and unimodal, and the privacy mechanisms are assumed to be Laplace or Gaussian. More complex settings can also satisfy the assumptions of Theorem 4.3, but details would need to be checked case-by-case. 

 Theorem \ref{thm:mleconsistency} offers insight into when we can expect the bounded and unbounded-DP MLEs to be similar, showing that as $\ndp$ becomes a more accurate approximation to $n$, the two MLEs converge. Since it may not be a common scenario for the privacy budget of $\ndp$ to diverge to infinity, this result is primarily of theoretical interest. Section \ref{s:MCEM} develops a computational approach that directly targets the unbounded-DP MLE. 

\begin{remark}
    The tightness condition ensures that $\hat{\theta}_{\epsilon}$ and $\hat{\theta}_0$ lie in compact sets with high probability. Alternatively, we could have assumed that the parameter space $\Theta$ is compact.

    Note that while we use $\epsilon$ as the \emph{privacy budget}, we do not necessarily require $\epsilon$-DP, or any other particular privacy definition, as long as the stated assumptions hold. 
\end{remark}

\section{Computational techniques}\label{s:computational}
To complement our theoretical results, which offer justification for some approximate plug-in procedures, in this section we develop 
computational methods to directly perform principled statistical inference on privatized data in the unbounded DP setting. In Section \ref{s:MCMC}, we develop a reversible jump MCMC for Bayesian inference and in Section \ref{s:MCEM}, we develop a Monte Carlo EM algorithm to estimate the MLE in unbounded DP. 

\subsection{A reversible jump MCMC sampler for unbounded DP}\label{s:MCMC}

The Metropolis-within-Gibbs sampler proposed by \citet{ju2022data} 
targets the bounded-DP private data posterior distribution $p(\theta \mid \sdp)$, with a runtime comparable to many non-private samplers. Their algorithm augments the parameter space with a latent database of fixed size $n$ and therefore is limited to bounded DP. Thus, besides plugging in $\ndp$ for $n$, it does not directly allow for inference on the unbounded-DP posterior distribution. 
While some unbounded DP problems can be expressed in a way that does not explicitly depend on $n$ (e.g., Section B of the Supplementary Materials), in general there is a need to develop trans-dimensional MCMC methods to perform Bayesian inference under unbounded DP. We propose a modification to \citet{ju2022data}'s sampler for unbounded DP, based on reversible jumps, which enjoys 1) lower bounds on the acceptance probabilities under $\ep$-DP and 2) either ergodicity or geometric ergodicity under regularity conditions. 

We aim to draw samples from the posterior distribution of $(n, \theta, x_{\{1:n\}})$ given the available information of $(\ndp,\sdp)$. Since it is not possible to sample directly from this posterior distribution, we rely on MCMC algorithms. However, we first notice that this posterior distribution is defined on a space that is the union of spaces with different dimensions. This is because if $n$ changes, the dimension of $(n, \theta, x_{\{1:n\}})$ also changes. Therefore, we need to rely on an MCMC algorithm that accounts for this aspect. For this reason, we use reversible jump MCMC \citep{green1995reversible}. To our knowledge, this is the first implementation of reversible jump MCMC techniques for inference on privatized data. 

For an introduction to reversible jump MCMC, we recommend \citet[Chapter 3]{brooks2011handbook}. Essentially, reversible jump MCMC  consists of ``within-model moves'' and ``between-model moves.'' In our case, the \emph{model} is determined by the sample size $n$.
\begin{itemize}
  \setlength\itemsep{0em}
    \item[-] \emph{Within-model move}: perform an  update to $(\theta, x_{\{1:n\}})$  
    using  one  cycle of the data augmentation MCMC method of \citet{ju2022data} with the current sample size $n$.
    \item[-] \emph{Between-models move}: propose to add/delete a row in the database, altering $n$ by 1.
\end{itemize}



We present a single round of our proposed RJMCMC sampler in Algorithm \ref{alg:RJMCMC}. When we cannot sample directly from $p(\theta|x_{1:n})$, another MCMC scheme with the correct target distribution can be substituted in this step, at the cost of slower mixing.

To ensure that the time to execute a single round of the RJMCMC is $O(n)$,  we assume that our privacy mechanism is \emph{record-additive} \citep{ju2022data}:
\begin{itemize}
  \setlength\itemsep{0em}
    \item A3 (Record Additivity) The privacy mechanism can be written in the form $p(\sdp \mid x) = g\left(\sdp, \sum_{i=1}^n t_i(x_i,\sdp)\right)$ for some known and tractable functions $g, t_i$.
\end{itemize}
Many popular mechanisms are record-additive, such as additive noise to summation statistics, as well as empirical risk minimization mechanisms, like objective perturbation \citep{Chaudhuri2011}. 
Record-additivity is important because it implies that when modifying a row of $x$, it only takes $O(1)$ time to update the value of $p(\sdp\mid x)$, by adding/subtracting at most two elements to $\sum_{i=1}^n t_i(x_i,\sdp)$. See Remark \ref{rem:recordadditivity} for more details.

In the case that the mechanisms for $s$ and $\ndp$ satisfy $\epsilon$-DP, we obtain lower bounds on the acceptance probabilities in Algorithm \ref{alg:RJMCMC}, similar to those in \citet{ju2022data}.

\begin{restatable}{prop}{propacceptance}\label{prop:acceptance}
Suppose that $s$ and $\ndp$ satisfy $\ep_s$-DP and $\ep_n$-DP, respectively, in the unbounded sense. The following hold for Algorithm \ref{alg:RJMCMC}: 
      1) the acceptance probability \eqref{eq:within} is lower bounded by $\exp(-2\epsilon_s)$, and 
      2) if $p(n)$ is a flat (improper) prior, then the acceptance probability \eqref{eq:between} is lower bounded by $\exp(-(\ep_s+\ep_n))$ for all $n\geq 2$; if $n=1$ and $n^*=2$, then the lower bound is instead $(1/2)\exp(-(\ep_s+\ep_n))$.
\end{restatable}

\begin{remark}\label{rem:recordadditivity}
    Similar to \citet{ju2022data},  record additivity allows us to ensure that one iteration of Algorithm \ref{alg:RJMCMC} only takes $O(n)$ time. Without record additivity, evaluating $p(s|x_{1:n})$ in \eqref{eq:within} and \eqref{eq:between} may take $O(n)$ time, giving a total runtime of $O(n^2)$ per cycle. More precisely, suppose that $g$, $t$, $p(n)$, $p(\ndp|n)$ and $p(x|\theta)$ can all be evaluated in $O(1)$ time, and $p(\theta|x_{1:n})$ can be sampled in $O(n)$ time.  Then, one iteration of Algorithm \ref{alg:RJMCMC} takes $O(n)$ time. 
\end{remark}

\begin{algorithm}[t]
\caption{One iteration of the privacy-aware Reversible-Jump MCMC sampler}
\label{alg:RJMCMC}
INPUT: Summaries $s$ and $\ndp$, functions $t$ and $g$ s.t. $p(s|x_{1:n})=g(s,\sum_{i=1}^n t(s,x_i))$, samplers for $p(\theta|x)$ and $p(x_i|\theta)$, initial dataset $x=(x_1,\ldots, x_n)$ and corresponding $t_x=\sum_i t(s,x_i)$.
\vspace{-.25cm}
\begin{enumerate}
  \setlength\itemsep{0em}
\item Within-model moves using Algorithm 1 of  \citet{ju2022data}:
\vspace{-.25cm}
\begin{enumerate}
  \setlength\itemsep{0em}
    \item Update $\theta\sim p(\theta|x)$.
    \item For $i = 1, \ldots, n$, sequentially update $x_i \mid x_{-i}, \theta, \sdp$ as follows:
    \begin{enumerate}
      \setlength\itemsep{0em}
      \vspace{-.2cm}
        \item Propose $x_i^* \sim p(x \mid \theta)$.
        \item Set $t_{x^*} = t_x-t(s,x_i)+t(s,x_i^*)$.
        \item Accept the move to $x_i=x_i^*$ with probability 
        \vspace{-0.25cm}
        \begin{equation}\label{eq:within}
   \min\left\{1,\frac{p(s|x^*)}{p(s|x)}\right\}=     \min\left\{1,\frac{g(s,t_{x^*})}{g(s,t_x)}\right\}.
        \end{equation}
    \end{enumerate}
\end{enumerate}
\vspace{-.25cm}
\item Between-models move: 
\vspace{-.5cm}
\begin{enumerate}
  \setlength\itemsep{0em}
\item Sample $n^*$ from the pmf 
$q(n^*|n)=\begin{cases}
    1&\text{if }n=1 \text{ and }  n^*=n+1\\
    1/2 & \text{if } n>1 \text{ and }    |n^*-n|=1\\
    0 & \text{otherwise}.
\end{cases}$
\vspace{-.5cm}
\begin{enumerate}
  \setlength\itemsep{0em}
        \vspace{-.1cm}
    \item If $n^*=n+1$,
    \begin{enumerate}
      \setlength\itemsep{0em}
        \item sample $x_{n+1}^*\sim p(x|\theta)$,
        \item set $x^*_{1:n^*} = (x_1,\ldots, x_n, x^*_{n+1})$
        and $t_{x^*} = t_x + t(s,x_{n+1}^*)$.
    \end{enumerate}
    \item Else if $n^*=n-1$,
    \begin{enumerate}
      \setlength\itemsep{0em}
        \item set $x^*_{1:n^*} = (x_1,\ldots, x_{n-1})$
        and $t_{x^*} = t_x - t(s,x_n)$ 
    \end{enumerate}
\end{enumerate}
\item Accept the move to $n=n^*$, $x_{1:n}= x^*_{1:n^*}$ and $t_x=t_{x^*}$ with probability
\end{enumerate}
\end{enumerate}
\begin{equation}\label{eq:between}
\hspace{-1.3cm}\min\left\{1,\frac{p(n^*) p(s|x^*)p(\ndp|n^*)q(n|n^*)}{p(n) p(s|x) p(\ndp|n) q(n^*|n)}\right\}=\min\left\{1,\frac{p(n^*) g(s,t_{x^*})p(\ndp|n^*)q(n|n^*)}{p(n) g(s,t_x) p(\ndp|n) q(n^*|n)}\right\}.
\end{equation}
 OUTPUT: $n$, $x_{1:n}$, $t_x$, $\theta$
\end{algorithm}

To ensure that a Markov chain has the correct limiting distribution and that empirical means from the chain have a Law of Large Numbers, it is important to establish that a chain is ergodic. In Proposition \ref{prop:ergodic} we give a sufficient condition to satisfy this property.
\begin{restatable}{prop}{propergodic}\label{prop:ergodic}
    Suppose that  
    1) the state space $\Theta\times \mscr X \times \mathbb N$ has a countably generated $\sigma$-algebra, 
    2) the following densities are bounded: $p(\theta)$, $p(x_i|\theta)$, and $p(s|x_{1:n})$, 
    3) for $i< j< k\in \NN$, if $p(n=i)>0$ and $p(n=k)>0$, then $p(n=j)>0$, 
    4) the supports of $p(x_i|\theta)$, $p(s|x_{1:n})$, and $p(\ndp|n)$ do not depend on the right-hand side. 
    Then, the RJMCMC of Algorithm \ref{alg:RJMCMC} is ergodic on the joint space $(\theta,x_{1:n},n)\in \Theta \times \mscr X\times \NN$, and has $p(\theta,x_{1:n},n|s,\ndp)$ as its unique limiting distribution. 
\end{restatable}

To quantify the estimation error of our chain with a Central Limit Theorem, geometric ergodicity is needed. In Proposition \ref{prop:geometric}, we offer sufficient conditions to ensure geometric ergodicity for Algorithm \ref{alg:RJMCMC}. The proof leverages 1) \citet{qin2025geometric}, who show that under mild conditions, if a chain is geometrically ergodic on within-model moves, and there are a finite number of models, then the reversible jump MCMC is also geometrically ergodic, and 2) \citet[Theorem 3.4]{ju2022data} to establish geometric ergodicity for the within-model moves.
 
\begin{restatable}{prop}{propgeometric}\label{prop:geometric}
Suppose that the conditions in Proposition \ref{prop:ergodic} hold, and further suppose 
1) $p(\theta)$ is a proper prior, 
2) $p(s|x)$ and $p(\ndp|n)$ satisfy $\ep_s$-DP and $\ep_n$-DP, respectively, 
3) there exists $a>0$ such that $p(x_i|\theta)>a$ for all $x_i$ and all $\theta$, and  
4) $p(n)$ has finite support. 
Then, the RJMCMC of Algorithm \ref{alg:RJMCMC} is geometrically ergodic. 
\end{restatable}

\subsection{Monte Carlo expectation maximization}\label{s:MCEM}
We develop a Monte Carlo expectation maximization (MCEM) algorithm to compute the MLE of $\theta$ given $(s,\ndp)$. Since the method does not explicitly depend on the separation between $s$ and $\ndp$, we will use $s$ as shorthand for the pair $(s,\ndp)$ in this section. 

This method iterates between the E-step, which estimates an expectation over the joint posterior distribution $p(x_{1:n},n|s,\theta^{(t)})$, approximated by our RJMCMC algorithm, and the M-step which maximizes the expectation calculated in the E-step or at least does a gradient step in the direction of maximization.  The method is summarized in Algorithm \ref{alg:MCEM}.

Since $x_{1:n}$ is unavailable due to privacy protection, the likelihood function 
\vspace{-.25cm}
$$\log p(\sdp\mid n,\theta) = \log \Bigg( \sum_{x_{1:n} \in \mscr X_1^n} p(\sdp \mid x_{1:n}) p(x_{1:n} \mid n, \theta) \Bigg) $$
\vspace{-.1cm}
is intractable due to the integration over $\mscr X_1^n$. Also, since $n$ is unknown, we additionally model $n$ with some (prior) distribution $p(n)$ and the marginal likelihood function is 
$$p(\sdp \mid \theta) =\sum_{n} p(n) p(\sdp \mid n, \theta) = \sum_{n}  p(n) \int_{x_{1:n}\in \mscr X_1^n} p(\sdp \mid x_{1:n}) p(x_{1:n} \mid n, \theta) dx.$$
We propose using an EM algorithm to estimate the MLE, $\hat{\theta} = \arg\max p(\sdp | \theta).$
The EM algorithm works with the complete data likelihood 
$ p(\sdp, x_{1:n}, n \mid \theta) = p(\sdp \mid x_{1:n}) p_{\theta}(x_{1:n} \mid n) p(n).$ 
In the E-step, the Q-function is the conditional expectation of complete-data log-likelihood given the current parameter estimate $\theta^{(t)}$:
\vspace{-.25cm}
\begin{align*}
    Q(\theta; \theta^{(t)}) &= \mathbb{E}_{x_{1:n},n \sim p(\cdot \mid \sdp,\theta^{(t)})}\left[\log p(\sdp,x_{1:n},n \mid \theta)\right] \\
    &= \sum_{n}\int_{\mscr X_1^n} \log p(x_{1:n}|\theta) p(x_{1:n}, n \mid \sdp, \theta^{(t)}) d x_{1:n}+ \text{const.}\\
    &= \mathbb{E}_{x_{1:n},n \sim p(\cdot \mid \sdp,\theta^{(t)})}\left[\sum_{i=1}^n \log p(x_i \mid \theta)\right]+ \text{const.}
\end{align*}
As we can see, the E-step involves the joint posterior distribution $p(x_{1:n}, n \mid \sdp, \theta^{(t)})$, which can be approximated with our reversible jump sampler. 
The MC approximation for $Q$ is
$$Q_{m}(\theta) = \frac{1}{m}\sum_{j=1}^{m} \log p(x^{(j)}_{1:n^{(j)}}|\theta)= \frac{1}{m} \sum_{j=1}^m \sum_{i=1}^{n^{(j)}} \log p(x_i^{(j)}|\theta),$$
where for $j=1,\ldots, m$, each $(x^{(j)}_{1:n^{(j)}},n^{(j)})$ is a sample from the joint posterior distribution $p(x_{1:n}, n \mid \sdp, \theta^{(t)})$, from the RJMCMC sampler.

In the M-step, we approximate $\arg\max Q(\theta)$ with $\arg\max Q_m(\theta)$, resulting in $\theta^{(t+1)} = \arg\max Q_m(\theta)$. 
For non-trivial models, the M-step may require numerical optimization.

Recall that it is sufficient to have $\theta^{(t+1)}$ such that $Q_{m}(\theta^{(t+1)}) > Q_{m}(\theta^{(t)})$ for convergence to a local maximum \citep{neal1998view}. 
As a result, 
we can instead use a gradient ascent step:
$\theta^{(t+1)} \gets \theta^{(t)} + \tau_t\sum_{j=1}^m \sum_{i=1}^{n^{(j)}} \partial_{\theta} \log p(x_i^{(j)}|\theta^{(t)}),$ where $\tau_t$ is the learning rate sequence. Such a gradient ascent step can be efficiently implemented for exponential family  models for the data $x_{1:n}$.
This gradient ascent step imposes another assumption: The E-step requires us to sample from the model $p(x|\theta)$ and evaluate its density, and the M-step  requires us to evaluate the gradient of the log-density.


\begin{algorithm}[t]
\caption{One iteration of MCEM}\label{alg:MCEM}
\begin{algorithmic}
\REQUIRE DP summary $s$, current values of $\theta^{(t)}$, sample size $n$, data set $x_{1:n}$, and learning rate sequence $\tau_t$.
\STATE 1. E-step: obtain the RJMCMC sample from the joint posterior distribution  $p(x_{1:n}, n \mid \sdp,\theta^{(t)})$.
\STATE  2. M-step: $\theta^{(t+1)} \gets \theta^{(t)} + \tau_t \sum_{j=1}^m \sum_{i=1}^{n^{(j)}} \partial_{\theta} \log p(x_i^{(j)}|\theta^{(t)}).$
\end{algorithmic}
 OUTPUT: $\theta^{(t+1)}$
\end{algorithm}

\begin{example}[Exponential Family] 

For an exponential family model, we have  $x_i \overset{iid}{\sim} p(x\mid \theta) = h(x) \exp( \eta(\theta) \cdot T(x) + A(\theta))$. 
To implement the M-step efficiently, 
one needs to solve the equation  
$ \nabla \eta(\theta) \left(\sum_j \sum_i T(x_i^{(j)})\right) = \nabla A(\theta).$ The gradient ascent step is  
$ \theta^{(t+1)} \gets \theta^{(t)} + \tau_t\left[\nabla \eta(\theta^{(t)}) \left(\sum_j \sum_i T(x_i^{(j)})\right) - \nabla A(\theta^{(t)}) \right].$
\end{example}

\begin{remark}
    Algorithm \ref{alg:MCEM} 
    can be easily modified to estimate the MLE in bounded DP: The E-step can be changed to a sampler for $p(x_{1:n}|\sdp,\theta^{(t)})$, such as the method in \citet{ju2022data}; the M-step is the same except that $n^{(j)}$ is replaced with the public value of $n$. 
\end{remark}

\section{Numerical experiments}\label{s:simulations}
We implement the RJMCMC and MCEM algorithms to explore their performance through a simulated linear regression setting as well as a real data example.

\subsection{Linear regression simulations }\label{s:linear-regression}
\begin{table}[t]
\centering

\begin{tabular}{lrrrrrr}
\hline
$\ep_n=$        & 0.001    & 0.01    & 0.1     & 1       & 10       & Inf      \\ \hline
$\EE(\beta_1)$  & -0.568   & -0.507  & -0.540  & -0.550  & -0.457   & -0.523   \\
$\var(\beta_1)$ & 5.559    & 1.024   & 1.053   & 0.942   & 0.936    & 0.879    \\
$\EE(\tau)$     & 1.052    & 1.057   & 1.025   & 1.042   & 1.026    & 1.052    \\
$\var(\tau)$    & 0.633    & 0.623   & 0.531   & 0.548   & 0.535    & 0.544    \\
$\EE(n)$        & 1116.219 & 987.717 & 998.707 & 999.874 & 1000.000 & 1000.000 \\
$\var(n)$       & 1005.443 & 762.425 & 145.776 & 1.924   & 0.005    & 0.000    \\ \hline
\end{tabular}

\vspace{.25cm}

\begin{tabular}{lrrrrrr}
\hline
$\ep_n=$        & 0.001    & 0.01     & 0.1     & 1       & 10       & Inf      \\ \hline
$\EE(\beta_1)$  & -0.921   & -0.942   & -0.952  & -0.946  & -0.946   & -0.956   \\
$\var(\beta_1)$ & 0.205    & 0.190    & 0.198   & 0.196   & 0.197    & 0.195    \\
$\EE(\tau)$     & 1.234    & 1.095    & 1.107   & 1.081   & 1.090    & 1.093    \\
$\var(\tau)$    & 0.351    & 0.291    & 0.296   & 0.274   & 0.286    & 0.293    \\
$\EE(n)$       & 1217.458 & 1000.015 & 999.109 & 999.871 & 1000.000 & 1000.000 \\
$\var(n)$       & 4051.367 & 258.795  & 66.134  & 1.665   & 0.005    & 0.000    \\ \hline
\end{tabular}  

\caption{Linear regression posterior mean and variance simulation results using $\ep_s=.1$ (top) or $\ep_s=1$ (bottom). Results are averaged over 100 replicates in each setting (one failed replicate for $\ep_n=.001$ with both $\ep_s=.1$ and $\ep_s=1$). 
Chains were run with 10,000 iterations, and used a burn-in of 5,000.
Note that $\ep_n=\mathrm{Inf}$ corresponds to bounded DP.}
\label{tab:sim01}
\end{table}

We analyze a linear regression problem under unbounded DP. 
The DP summary statistics are similar to those used in \citet{ju2022data}, but unlike their example, we use a full prior model, whereas they assumed that the values of many nuisance parameters were known. We model the original data as $y_i|x_i,\beta,\tau \sim N((1,x_i)\beta,\tau^{-1}I)$ and $x_i|\mu,\Phi \sim N_p(\mu,\Phi^{-1})$. We use priors $\beta|\tau \sim N_{p+1}(m,\tau^{-1}V^{-1})$, $\tau \sim \mathrm{Gamma}(a/2,b/2)$ (shape-rate parametrization), $\Phi\sim \mathrm{Wishart}_p(d,W)$, and $\mu\sim N(\theta,\Sigma)$.  
The parameters are $(\beta, \tau, \mu, \Phi)$,  
the hyperparameters are $(m, V, a, b, \theta, \Sigma, d, W)$,
and sufficient statistics are $\left(X^\top X, X^\top Y, Y^\top Y\right)$, where $X$ is a matrix whose rows are $(1,x_i)$ and $Y$ is a vector of the $y_i$'s.

Our DP summary statistics are produced by the following procedure. We assume that the data curator has  assumed bounds on the $x_{i,j}$'s and $y_i$'s,  $[L,U]$ (while $L$ and $U$ could vary based on the covariate/response, we hold them fixed here for simplicity). 
 The mechanism first clamps the data to lie within the pre-specified bounds and normalizes the $x_{i,j}$'s and $y_i$'s to lie in $[-1,1]$ using the following map: $f(x;L,U) = \frac{2([x]_L^U-L)}{U-L}-1$, where $[x]_L^U = \min\{U,\max\{x,L\}\}$ is the clamping function. Call $\twid x_{i,j}$ and $\twid y_i$ the output of this clamping-normalization procedure and call $\twid X$ the matrix whose rows are $(1,\twid x_{i})$ and $\twid Y$ the vector of the $\twid y_i$'s.  
After this, we add independent Laplace noise to each of the unique entries of $\twid X^\top\twid X$, $\twid X^\top \twid Y$, and $\twid Y^\top \twid Y$, excluding the $(1,1)$ entry of $\twid X^\top\twid X$, which encodes the sample size $n$. 
Using unbounded DP, the $\ell_1$ sensitivity of these unique entries is $\Delta = p^2/2+(5/2)p+2$, where $p$ is the number of covariates  (see Section C.1 of the Supplementary Materials for a derivation). 
Call $s$ the result of adding independent $\mathrm{Laplace}(0,\Delta/\ep_s)$ noise to the unique entries of $\twid X^\top\twid X$, $\twid X^\top \twid Y$, and $\twid Y^\top \twid Y$, excluding the $(1,1)$ entry of $\twid X^\top\twid X$.  
We privatize $n$ as $\ndp = n + \mathrm{Laplace}(1/\ep_n)$.  Note that clamping is used to calculate the summaries, which serve as the location parameters for the Laplace privacy mechanism $\eta$, whereas the data model $p(\cdot|\theta)$ is for the original, unclamped data, which are normally distributed. By the CLT, summaries $s$ satisfy Assumption A1 with parameters $a=1$ and $b=1/2$, since $s$ consists of sums of bounded and independent entries and, after scaling by $n^{-1}$, the additive noise is of order $O_p(1/n)$. A2 holds by Lemma A.2, since $(n-n_0|\ndp=n_0)=O_p(1/\ep)=o_p(n^{1/2})=o_p(n_0^{1-b})$, provided that $\epsilon=\omega(n^{-1/2})$. A3 holds since the statistics are sums.

We use the following hyperparameters: $L=-5$, $U=5$, $p=2$, $m=(0,0,0)^\top$, $V=I_{p+1}$, $a=2$, $b=2$, $\theta=(0,0)^\top$, $\Sigma=I_p$, $d=2$, $W=I_p$. We use $\ep_s\in\{.1,1\}$, and vary $\ep_n \in \{.001,.01,.1,1,10,\text{Inf}\}$, where $\ep_n=\text{Inf}$ means that $n$ is released without noise (effectively $2\ep_s$-DP under bounded DP). 
We use a sample size of $n = 1000$ and run our RJMCMC for 10,000 iterations with a burn-in of 5,000. Each replicate uses the same dataset but includes new privacy noise and a new Markov chain. Data were generated under the model
$y_i \mid x_i, \beta, \tau \sim N((1, x_i)\beta, \tau^{-1} I)$
and
$x_i \mid \mu, \Phi \sim N_p(\mu, \Phi^{-1})$,
with $\beta = (\beta_0, \beta_1, \beta_2)^\top = (0, -1, 1)^\top$, $\tau = 1$, $\mu = (-1, 1)^\top$, and $\Phi=I_2$. There was one failed chain in both settings of $\ep_n=.001$, which were omitted.  Table \ref{tab:sim01} contains  the posterior mean and posterior variance for $\beta_1$, $\tau$, and $n$, averaged over the 100 replicates. Results for other parameters, as well as standard errors are in Tables 3 and 4 in the Supplementary Materials. 

\begin{table}[t]
\centering
\begin{tabular}{l||c|c|c|c}
  \hline
 $\ep_n$& $\beta_0$ & $\beta_1$ & $\beta_2$ & $\tau$ \\
 \hline
 0.001 & -0.400 (0.751) &  -0.977 (0.579) &  0.770 (0.702)   &  1.17   (0.710)\\
 0.01 & -0.192 (0.433) & -0.914 (0.408) & 0.847  (0.390) &  1.01    (0.426)\\
 0.1 & -0.111 (0.374) & -0.968 (0.345) & 0.871 (0.365) &  1.14    (0.562) \\
 1 & -0.115 (0.369) & -0.952 (0.338) & 0.883 (0.387) &   1.17    (0.664)\\

 10 & -0.131 (0.435) &-0.941  (0.352) & 0.875  (0.465) &  1.15  (0.571)\\
 Inf &  -0.167 (0.434) & -0.955 (0.385) & 0.828 (0.483) &    1.14    (0.554)\\ 
  \hline
\end{tabular}
\caption{Average and standard deviation (in parentheses) of MLE $\hat{\beta}$ and $\hat{\tau}$ values for linear regression. For a fixed dataset of size $n = 1000$, we generate $(\sdp,\ndp)$ using $\ep_s=1$ and varying $\ep_n$. Results are over 100 replicates. MCEM chains were run with 10,000 iterations, and used a burn-in of 30\%.
Note that $\ep_n=\mathrm{Inf}$ corresponds to bounded DP.}
\label{tab:mcem}
\end{table}

In the top of Table \ref{tab:sim01}, we have results for $\ep_s=.1$, and see that the posterior means and variances of $\beta_1$ and $\tau$ are very similar for $\ep_n\geq .01$. At $\ep_n=.001$ we have an inflated posterior variance for $\beta_1$. However,  we do have higher uncertainty about $n$ when using $\ep_n\leq .1$. In the bottom of Table \ref{tab:sim01}, the results are even more stable for $\ep_s=1$. As in the $\ep_s=.1$ case, the greatest uncertainty is in $n$ which appears for $\ep_n\leq .1$. Interestingly, by increasing $\ep_s$ from .1 to 1, we see a reduction in the posterior variance of $n$ even when $\ep_n$ is fixed. This illustrates how the Bayesian posterior can learn about $n$ from other summaries. 

In Table \ref{tab:mcem} we conduct a similar experiment, using the MCEM algorithm to estimate the parameters $\beta_0$, $\beta_1$, $\beta_2$, and $\tau$, where $\ep_s=1$.  We use the closed-form solutions for the M-step, which are the non-private MLEs for the normal model and linear regression. While we were not able to verify the assumptions in Theorem \ref{thm:mleconsistency}, we see that the MCEM algorithm also offers accurate estimates. As in Table \ref{tab:sim01}, the estimates are similar for $\epsilon_n \ge 0.1$ while the estimates at privacy level $\epsilon_n = 0.001$ have higher bias and more uncertainty.  



To our knowledge, no existing DP methods handle unknown sample size in a principled way. For comparison, we use a version of the parametric bootstrap \citep{ferrando2022parametric}, adapted by \cite{barrientos2024feasibility}, which uses a plug-in DP estimate of $n$, needed to estimate $\tau^{-1}$.  The results are included in Section C.2 of the Supplementary Materials, and show  unstable results, especially with small privacy parameters.

\begin{figure}[ht!]
\vspace{-.75cm}
    \centering
    \includegraphics[width=.48\linewidth, page = 1]{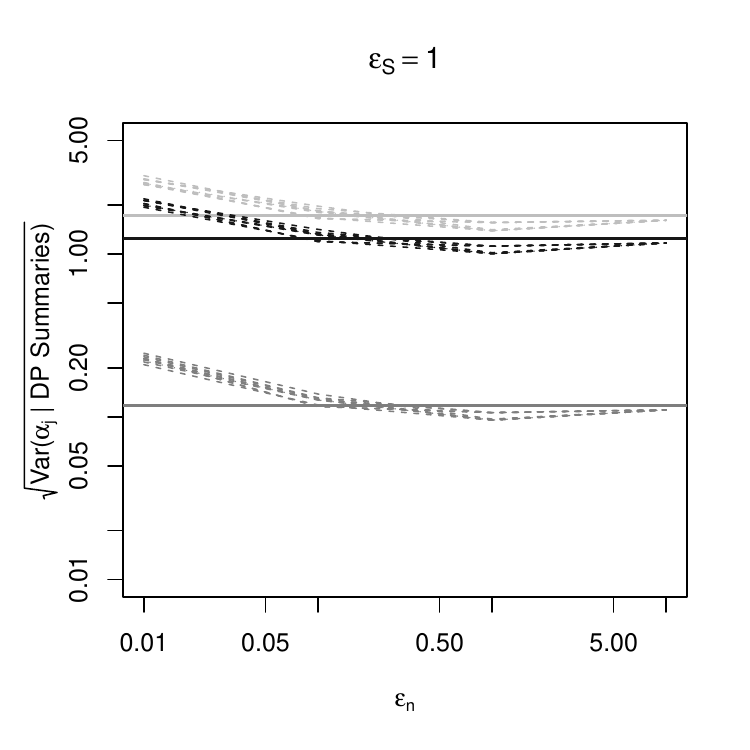}
    \includegraphics[width=.48\linewidth, page = 2]{ATUS_figures.pdf}

\vspace{-.7cm}
    
    \includegraphics[width=.48\linewidth, page = 4]{ATUS_figures.pdf}
    \includegraphics[width=.48\linewidth, page = 5]{ATUS_figures.pdf}
    
    \vspace{-.6cm}
    \caption{ATUS application: First column shows posterior standard deviations for $\alpha_1$, $\alpha_2$, and $\alpha_3$ with varying $\epsilon_{s}$ and $\epsilon_{n}$. Horizontal solid lines represent values under bounded DP, while the 10 dashed lines represent standard deviations for 10 realizations of $\ndp$. Second column shows box plots for the posterior distribution of $n$ under unbounded DP, with the horizontal line indicating the actual sample size. Each box plot combines draws for $n$ from 10 posterior distributions, each corresponding to one of the 10 realizations of $\ndp$ .}
    \label{fig:ATUS1}
\end{figure}

\subsection{Real data application}\label{s:realData}

We analyze data from the publicly available 2019 American Time Use Survey (ATUS) Microdata File, which can be accessed at \href{https://www.bls.gov/tus/datafiles-2019.htm}{https://www.bls.gov/tus/datafiles-2019.htm}. The ATUS dataset, collected and maintained by the U.S. Bureau of Labor Statistics, records the daily time allocation for various activities (such as personal care, household tasks, etc.) throughout 2019. The dataset is comprised of 9,435 records, released on July 22, 2021.


\cite{guo2024differentially} introduce differentially private Bayesian methods for analyzing compositional data, using the ATUS dataset. They focus on proportions of daily time spent on personal care ($x_{i1}$), eating and drinking ($x_{i2}$), and all other activities ($x_{i3}$), resulting in 6,656 observations after preprocessing. They model the data with a Dirichlet distribution with parameters $(\alpha_1, \alpha_2, \alpha_3)$. We aim to compare their bounded DP approach with our unbounded one, using the same prior specification: $\alpha_j \overset{i.i.d.}{\sim}\mathrm{Gamma}(1.0, 0.1)$ for $j = 1, 2, 3$. 

For the  experiment, we work with fixed realizations of $\sdp$, each one generated using a privacy budget $\epsilon_{s} \in \{1, 10\}$ and study how the results vary with different realizations of $\ndp$. 
We define
$
\sdp = \sum_{i=1}^n \left[\log\left([x_{i1}]_a^1\right), \log\left([x_{i2}]_a^1\right), \right.$ $\left. \log\left([x_{i3}]_a^1\right) \right] + L_1,
$
where $L_1 \sim \mathrm{Laplace}(0, -3 \log(a)/\ep_s)$, setting $a = \min_{ij} x_{ij} = 0.0006$ as if an oracle had provided this value.  For $\ndp$, we use a privacy budget of $\epsilon_{n} \in \{0.01, 0.1, 1, 10\}$ and generate 10 realizations using the Laplace mechanism with sensitivity equal to $1$. To sample from the posterior distribution for $(\alpha_1, \alpha_2, \alpha_3)$ under bounded DP, we utilize the algorithm outlined in \cite{ju2022data}, while for unbounded DP, we employ our proposed RJMCMC algorithm.

Figure \ref{fig:ATUS1} presents the obtained results. 
In the first column of Figure \ref{fig:ATUS1}, we observe that the posterior standard deviation of the parameters under unbounded DP decreases in $\epsilon_n$, converging to those under bounded DP. Additionally, as $\epsilon_s$ increases, the results become more accurate, but the value of $\epsilon_n$ required for unbounded DP to match its benchmark must also increase. 
In the second column of Figure \ref{fig:ATUS1}, we show the posterior distribution of $n$ under unbounded DP. As expected, we see that when $\epsilon_n$ increases, the uncertainty about $n$ decreases. Interestingly, in the case of $\epsilon_n=.01$, we observe that the variability of the posterior distribution of $n$ slightly decreases as $\epsilon_s$ goes from 1 to 10, showing that there is some additional learning about $n$ when $\sdp$ becomes more accurate.  This result is consistent with the results of Section \ref{s:linear-regression}. Additional results for this example are provided in Figures C2 and C3 in the supplementary materials.


\section{Discussion}\label{s:discussion}
 In this paper, we studied how frequentist and Bayesian distributions differ between the known and unknown sample size settings. Our theoretical analyses offer some valuable insights: While it is common in the DP literature to equally split the privacy loss budget between statistics, our results in Section \ref{s:sampleInf} indicate that we can use a vanishing amount of the privacy loss budget to estimate the sample size $n$, and still asymptotically obtain the same sampling distribution. Furthermore, the necessary rate of $n_{dp}-n$ varies depending on the nature of the asymptotic distribution of the other DP summary $s$, which could impact the choice of mechanism for $s$ (see Example \ref{ex:LaplaceKNG}). Without any known asymptotic structure about $s$ (such as A1), we have a different phenomenon: the privacy parameter of $s$ must go to zero faster than the privacy parameter of $\ndp$. While our theoretical results offer valuable insight, they may not hold in all settings, and in the case of Bayesian analysis are only able to discuss ABC-type posteriors. For these reasons, we also developed new
 numerical algorithms to enable valid statistical inference under unbounded DP. 
 As predicted by our theoretical results in Section \ref{s:sampleInf}, our numerical results also show that the cost of unbounded DP is typically small for a wide range of privacy parameters for $\ndp$.



 While our focus is on unbounded DP, the proposed algorithms may also be useful in bounded DP settings. For example, when analyzing a subset of a dataset under bounded DP, the sample size of the subset may itself be private. In such cases, with relatively straightforward adjustments, our approach can be applied to account for the randomness introduced when using a privatized version of that sample size.


Our Bayesian theory as $n_0\rightarrow \infty$ is limited in that it only considers ABC-type posterior distributions. It would be more ideal to analyze the posterior distribution that directly conditions on the observation itself, rather than this type of event. Potentially, the theory of the convergence of ABC posteriors could be used to bridge this gap \citep{frazier2018asymptotic}.

Another limitation of our theory is that the convergence of the MLE is only with respect to the privacy budget. We suspect that the MLE is also a consistent estimator as $n\rightarrow\infty$, but other analysis techniques are required to establish this claim.  To our knowledge, asymptotic properties of the DP MLE have not been studied even under bounded DP.



  While this paper was focused on how frequentist and Bayesian distributions differ when the sample size is known or unknown, a related question is the fundamental tradeoff due to the unbounded framework versus the bounded DP framework. To our knowledge, prior DP minimax results are in the bounded DP framework \citep{Barber2014,cai2019cost,cai2020cost}. Since unbounded DP implies bounded DP (up to a factor in the privacy parameter), if an unbounded mechanism has a rate that matches the bounded DP lower bound, then this is sufficient to establish that the unbounded DP rate is the same as the bounded DP rate. However, we are not aware of prior results that show either a gap between the bounded and unbounded DP rates, or prove that the rates are the same.

  \section*{Acknowledgments}
  The authors used ChatGPT 5.5, to debug, organize, and comment some parts of the code, as well as to proofread the text. All content has been carefully reviewed by the authors.

\bibliographystyle{apalike}
\bibliography{bibliography}

\appendix

In this Supplementary Materials document, we include proofs and technical details in Section \ref{s:proofs}, an example of privatized multinomial data in Section \ref{s:multi}, and additional numerical results in Section \ref{s:extraSim}.
 \section{Proofs and technical lemmas }\label{s:proofs}
 \subsection{Conditions for assumptions to be satisfied}
\begin{restatable}{lem}{lemflat}\label{lem:flat}
Let $p(n)$ be a flat prior on $\{1,2,\ldots\}$. Suppose that the privacy mechanism $\eta(\ndp|n)\defeq p(\ndp|n)$ is a pmf where $n\in \NN$ and $\ndp\in \ZZ$, and satisfies $\eta(n_1|n_2)=\eta(n_2|n_1)$ for all $n_1,n_2\in \NN$. Then if $n_0\in \mathbb{N}$ and $(\ndp-n_0|n=n_0)=O_p(a_{n_0})$, where $a_{n_0}=o(n_0)$, then $(n-n_0|\ndp=n_0)=O_p(a_{n_0})$ as well. 
\end{restatable}
\begin{proof}
Let $n_0\in \NN$. The posterior pmf of $n|\ndp=n_0$ is 
\begin{align*}
    p(n|\ndp=n_0)&=\frac{p(\ndp=n_0|n)p(n)}{\sum_{i=1}^\infty p(\ndp=n_0|n=i)p(n=i)}\\&=\frac{\eta(n|n_0)}{\sum_{i=1}^\infty \eta(i|n_0)}\\
    &\geq \eta(n|n_0),
\end{align*}
where we used the symmetry of $\eta$, and the observation that $\sum_{i=1}^\infty \eta(i|n_0)=P(\ndp\geq 1|n=n_0)\leq 1$. 

Let $\ep>0$ be given. Since $(\ndp-n_0|n=n_0)=O_p(a_{n_0})$, there exists $M>0$ and $N_1$ such that for all $n_0\geq N$, $P(|\ndp-n_0|\leq Ma_{n_0}|n=n_0)\geq1-\ep/2$. Since $a_{n_0}=o(n_0)$, there exists $N_2$ such that for all $n_0\geq N_2$, $P(|\ndp-n_0|\leq n_0-1|n=n_0)\geq 1-\epsilon/2$. Set $N_3=\max\{N_1,N_2\}$. Then for all $n_0\geq N_3$, we have 
\begin{align*}
P(|n-n_0|\leq M a_{n_0}|\ndp=n_0) &= \sum_{\substack{i \text{ s.t. }|i-n_0|\leq Ma_{n_0}\\\text{and }i\geq 1}} p(n=i|\ndp=n_0)\\
&\geq \sum_{\substack{i \text{ s.t. }|i-n_0|\leq Ma_{n_0}\\\text{and }i\geq 1}}\eta(i|n_0)\\
&=P(|\ndp-n_0|\leq Ma_{n_0} \text{ and } \ndp\geq 1| n=n_0)\\
&\geq  P(|\ndp-n_0|\leq Ma_{n_0} | n=n_0)+P(\ndp\geq 1|n=n_0)-1\\
&\geq  P(|\ndp-n_0|\leq Ma_{n_0} | n=n_0)\\
&\phantom{\geq}+P(|\ndp-n_0|\leq n_0-1|n=n_0)-1\\
&\geq (1-\ep/2)+(1-\ep/2)-1\\
&\geq 1-\ep.
\end{align*}
We conclude that $(n-n_0|\ndp=n_0)=O_p(a_{n_0})$. 
\end{proof}

\begin{lem}
Let $p(n)$ be a flat prior on $\{1,2,\ldots\}$. Suppose that $\ndp|n \sim \mathrm{Laplace}(0,1/\epsilon)$ where $0<\epsilon\leq \epsilon_0$. Then, $n-n_0|\ndp=n_0 = O_p(1/\epsilon)$ as $n_0\rightarrow \infty$.    
\end{lem}
\begin{proof}
    We see that the pmf of $n|\ndp=n_0$ is 
    \[p(n|\ndp=n_0) = \frac{\exp(-\epsilon |n-n_0|)}{\sum_{i=1}^\infty \exp(-\epsilon|i-n_0|)},\] 
    with support $n\in \NN$. 
    Define also $\twid n|n_0$ as the random variable with pmf 
    \[p(\twid n|n_0) = \frac{\exp(-\epsilon|\twid n-n_0|)}{\sum_{i=-\infty}^\infty \exp(-\epsilon |i-n_0|)},\]
with support $\twid n\in \ZZ$. We see that $n|\ndp=n_0 \overset d = \twid n|\{\ndp=n_0 \& \twid n\geq 1\}$, since $n$ and $\twid n$ have proportional pmfs, but with different support. Thus, we will be able to approximate probabilities about $n|\ndp=n_0$ with those about $\twid n|n_0$.

We need to get a lower bound on the series $\sum_{i=-\infty}^\infty \exp(-\ep|i-n_0|)$. Call $a=n_0-\lfloor n_0\rfloor\in [0,1)$. Then,
\begin{align*}
    \sum_{i=-\infty}^\infty \exp(-\ep|i-n_0|) &= \sum_{i=-\infty}^\infty \exp(-\ep|i-a|)\\
    &=\sum_{i=-\infty}^0\exp(-\ep(a-i))+\sum_{i=1}^\infty \exp(-\ep(i-a))\\
    &=\frac{\exp((1-a)\ep)+\exp(a\ep)}{\exp(\ep)-1}.
\end{align*}
To minimize this expression over $a\in [0,1)$, we set the derivative equal to zero, which results in the solution $a=1/2$. We also need to check the end points $a=0$ and $a=1$. At $a=1/2$, the numerator has value $2\exp(\ep/2)$ and at either $a=0$ or $a=1$, the numerator has value $\exp(\ep)+1$. Since $\exp(x)$ is a convex function, we have that $\exp(\ep/2)\leq (\exp(\ep)+1)/2$, which establishes that the minimizer is at $a=1/2$. Thus, we have 
\[\sum_{i=-\infty}^\infty \exp(-\ep|i-n_0|)\leq \frac{2\exp(\ep/2)}{\exp(\ep)-1}.\]
So, we have an upper bound on the pmf for $\twid n$:
\[p(\twid n|n_0)\leq \left(\frac{\exp(\ep)-1}{2\exp(\ep/2)}\right)\exp(-\ep|\twid n - n_0|).\]
Next, we a lower bound on the probability that $\twid n\geq 1$:
\begin{align*}
P(\twid n \geq 1|n_0)=1-P(\twid n \leq 0|n_0)&\leq \left(\frac{\exp(\ep)-1}{2\exp(\ep/2)}\right)\sum_{i=-\infty}^0 \exp(-\ep(n_0-i))\\
&=\left(\frac{\exp(\ep)-1}{2\exp(\ep/2)}\right)\left(\frac{\exp(-\ep-n_0 \ep)}{\exp(\ep)-1}\right)\\
&\leq \frac{1}{2\exp(\ep/2)}\\
&\leq 1/2,
\end{align*}
where we used the fact that $n_0\geq 1$ and $\ep>0$

We are ready to show that $(n-n_0|\ndp=n_0)=O_p(1/\epsilon)$, provided that $0<\ep\leq \ep_0$. Let $\gamma >0$ be given and set $M=\log(2/\gamma)+\ep_0/2$. Then, for all $n_0\geq 1$, we have that
\begin{align}
&P(|n-n_0|\ep\geq M|\ndp=n_0)\\
&= \frac{P(|\twid n-n_0|\ep\geq M \ \&\  \twid n\geq 1)}{P(\twid n\geq 1)}\\
&\leq 2P(|\twid n-n_0|\geq M/\ep|n_0)\label{eq:cond}\\
&=2\left(\frac{\exp(\ep)-1}{2\exp(\ep/2)}\right) \left[\sum_{-\infty}^{\lfloor n_0-M/\ep\rfloor}\exp(-\ep(n_0-i)) +   \sum_{i=\lceil{n_0+M/\ep}}^\infty \exp(-\ep(i-n_0))\right]\\
&=\left(\frac{\exp(\ep)-1}{\exp(\ep/2)}\right) \left[\sum_{j=-\infty}^{\lfloor a-M/\ep\rfloor} \exp(-\ep(a-j)) + \sum_{j=\lceil a+M/\ep\rceil}^\infty \exp(-\ep(j-a))\right]\label{eq:j}\\
&=\left(\frac{\exp(\ep)-1}{\exp(\ep/2)}\right) \left[\frac{\exp(\ep(-\lceil a+M/\ep\rceil+a+1))+\exp(\ep(\lceil a-M/\ep\rceil -a+1))}{\exp(\ep-1)}\right]\\
&\leq \frac{\exp(\ep[-(a+M/\ep)+a+1])+\exp(\ep[a-M/\ep-a+1])}{\exp(\ep/2)}\label{eq:round}\\
&=\frac{\exp(\ep[1-M/\ep])+\exp(\ep[1-M/\ep])}{\exp(\ep/2)}\\
&=2\exp(\ep/2)\exp(-M)\\
&\leq 2\exp(\ep_0/2)\exp(-M)\label{eq:ep0}\\
&\leq \gamma,\label{eq:plugin}
\end{align}
where line \eqref{eq:cond} upper bounds the probability of intersection with the probability of the single event and lower bounds the denominator using $P(\twid n\geq 1|n_0)\geq 1/2$, line \eqref{eq:j} re-indexes with $j=i-\lfloor n_0\rfloor$, \eqref{eq:round} uses the lower bounds $(a+M/\ep)\leq \lceil a+M/\ep\rceil$ and $(a-M/\ep) \geq \lfloor a-M/\ep\rfloor$,  \eqref{eq:ep0} uses the assumption that $\ep\leq \ep_0$, and \eqref{eq:plugin} plugs in the value $M=\log(2/\gamma)+\ep_0/2$ and simplifies the expression. We see that $n-n_0|\ndp=n_0=O_p(1/\ep)$ as $n_0\rightarrow \infty$. 
\end{proof}
    
The result below is a useful technical lemma which is similar to Slutsky's theorem for convergence in distribution. It is crucial to extend the results of Theorem \ref{thm:convergence} to the setting of convergence in total variation. 
\begin{lem}
    [\citealp{parthasarathy1985tool}]\label{lem:parthasarathy}
    Suppose that $X_n \overset {\mathrm{TV}}\rightarrow X$, where $X\in \RR^d$ is an absolutely continuous random variable, $Y_n\overset{d}\rightarrow Y$ where $Y\in \RR^d$, and $c_n \rightarrow c\neq 0$. Then, $c_n X_n+Y_n \overset {\mathrm{TV}}\rightarrow cX+Y$.
\end{lem}

We also include two useful methods to assert convergence in total variation central limit style quantities. Lemma \ref{lem:clt_cont} is an extension of an old result showing convergence in total variation in the one-dimensional case \citep{parthasarathy1985tool}.

\begin{lem}[\citealp{bally2016asymptotic}]\label{lem:clt_cont}
    Let $X_i$ be i.i.d. random variables in $\RR^d$ with mean zero and finite covariance matrix $C(X)$. Call $S_n = n^{-1/2} C^{-1/2}(X)\sum_{i=1}^n X_i$. Then, $\mathrm{TV}(S_n, N(0,I))\rightarrow 0$ if and only if there exists $n_0\geq 1$ such that $S_{n_0}$ has an absolutely continuous component: $S_{n_0} \overset d = \chi V+(1-\chi)W$ where $\chi\sim \mathrm{Bern}(p)$ for $p>0$ and $V$ absolutely continuous. 
\end{lem}

A consequence of Lemma \ref{lem:clt_cont} is that it is impossible for a sample mean of discrete random variables to converge in total variation to a Gaussian. However, recently it was shown that for integer-valued random variables, the sample mean does approach a discretized Gaussian \citep{gavalakis2024entropy}; in fact this was shown to be equivalent to adding an appropriate uniform random variable to the sum of discrete RVs, which then converges in total variation to a continuous Gaussian.

\begin{lem}[\citealp{gavalakis2024entropy}]\label{lem:clt_disc}
    Let $X_i$ be i.i.d. random variables taking values in $\ZZ$, which have mean $\mu$ and finite variance $\sigma^2$, and  let $S_n = n^{-1/2} \sum_{i=1}^n X_i$. Let $Z\sim N(\mu,\sigma^2)$. Then, $\mathrm{TV}(S_n,n^{-1/2}\lfloor n^{1/2}Z\rfloor)\rightarrow 0$ and $\mathrm{TV}(S_n+U,Z)\rightarrow 0$, where $U\sim \mathrm{Unif}(-n^{-1/2}/2,n^{-1/2}/2)$. 
\end{lem}

In the following example, we summarize a few common DP settings where one may be able to ensure that assumption A1' holds, leveraging the lemmas above.  

\begin{lem} [Central limit theorem examples for A1']\label{lem:tv_clt}
    In the following, assume that $N$ is an absolutely continuous random variable with mean zero and finite variance, and that $t(x_i)$ has finite mean and finite variance as well. Then, under any of the following settings, we have that $s$ satisfies assumption A1':
    \begin{enumerate}
        \item Statistic with an absolutely continuous component: $s=\sum_{i=1}^{n_0}t(x_i)+N$, where $t(x_i)\in \RR^d$ has an absolutely continuous component: $t(x_i) \overset d = \chi v + (1-\chi) w$, where $\chi\sim \mathrm{Bern}(p)$ for $p>0$, $v$ is absolutely continuous,  $w$ is any random variable, and $N$ is absolutely continuous.  For example, if $t(x_i) = [\twid t(x_i)]_l^u$ where $\twid t(x_i)$ is absolutely continuous, then $t(x_i)$ has an absolutely continuous component.
        \item Locally private additive mechanism with continuous noise: $s=\sum_{i=1}^{n_0}(t(x_i) + N_i)$, where $N_i$ is absolutely continuous  and $t(x_i)$ is any random variable.
        \item Integer-valued statistic with uniform noise: $s=\sum_{i=1}^{n_0} t(x_i) + U+N$, where $t(x_i)\in \ZZ$ and  $U\sim \mathrm{Unif}(-1/2,1/2)$. For example, when $N$ is a discrete Laplace random variable \citep{Inusah2006}, then $U+N$ has the truncated-uniform Laplace (tulap) distribution \citep{awan2018differentially} which satisfies $\ep$-DP. 
    \end{enumerate}
\end{lem}
\begin{proof}
    \begin{enumerate}
        \item This result follows from Lemmas \ref{lem:clt_cont} and \ref{lem:parthasarathy}.
        \item The convolution of an absolutely continuous random variable with another random variable  (possibly discrete) is absolutely continuous \citep{millier2016density}; the result follows from Lemma \ref{lem:clt_cont} 
        \item This result follows from Lemmas \ref{lem:clt_disc} and \ref{lem:parthasarathy}.
    \end{enumerate}
\end{proof}

\subsection{Main technical results}

\begin{restatable}{lem}{lempolya}[Multivariate Polya's Theorem]\label{lem:polya}
If $X_n\overset d\rightarrow X$, and $X\in \RR^d$ is a continuous random vector, then the convergence of the multivariate cdfs is uniform. It follows that $\mathrm{KS}(p(x_n),p(x))\rightarrow 0$. 
\end{restatable}
\begin{proof}
    Let $\Phi$ be the cdf of $N(0,1)$, and let $Y=\Phi(X)$, and $Y_n=\Phi(X_n)$, where $\Phi$ is applied elementwise. Since $\Phi$ is continuous and invertible, we have that $X_n \overset d \rightarrow X$ if and only if $Y_n \overset d \rightarrow Y$. We also have that $F_{Y}(t) = F_{X_n}(\Phi^{-1}(t))$. We see that $F_{X_n}$ converges uniformly to $F_X$ if and only if $F_{Y_n}$ converges uniformly to $F_Y$. So, we will show the result for $Y_n$. 

    For $k\geq 1$, let $\mscr B_k^d$ be the set of regions, which decompose $[0,1]^d$ into $2^{dk}$ disjoint, equally sized cubes   with endpoint coordinate values in $\{j/2^k|j=0,1\ldots 2^k\}$ (e.g., $\mscr B_1^2 = \{[0,1/2]\times [0,1/2], [0,1/2]\times[1/2,1], [1/2,1]\times [0,1/2], [1/2,1]\times[1/2,1]\}$). Note that the $\mscr B_k^d$ are nested in the sense that for each $B_1\in \mscr B_k^d$, there exists $B_2\in \mscr B_{k+1}^d$ such that $B_1\supset B_2$. For $B\in \mscr B_k^2$, call $U^B$ the vector with $U^B_i=\max_{t\in B} t_i$ and $L^B$ the vector with $L^B_i=\min_{t\in B}t_i$. Note that $U^B$ can be viewed as the upper right corner of $B$ and $L^B$ is the lower left corner of $B$. Call
    \begin{align*}\gamma_k &= \max_{B\in \mscr B_k^d}[\max_{t\in B} F_Y(t)-\min_{t\in B}F_Y(t)]\\
    &=\max_{B\in \mscr B_k^d} [F_Y(U^B)-F_Y(L^B)],
    \end{align*}
    by the monotonicity of $F_Y$. Note that $\gamma_k$ is a decreasing sequence which converges to zero,  since $\mscr B_k^d$ becomes more dense in $[0,1]^d$ and are nested, and by the continuity of $F_Y$. 

    Let $x\in [0,1]^d$ be arbitrary, and let $B_k^x\in \mscr B_k^d$ be the cube which contains $x$. Then 
    \begin{align*}
        F_{Y_n}(x)-F_Y(x)&\leq F_{Y_n}(U^{B_k^x})-F_Y(L^{B_k^x})\\
        &=F_{Y_n}(U^{B_k^x})-F_Y(U^{B_k^x})+F_Y(U^{B_k^x})-F_Y(L^{B_k^x})\\
        &\leq F_{Y_n}(U^{B_k^x})-F_Y(U^{B_k^x})+\gamma_k.
    \end{align*}
    Similarly,
    \begin{align*}
        F_{Y_n}(x)-F_Y(x)&\geq F_{Y_n}(L^{B_k^x})-F_Y(U^{B_k^x})\\
        &=F_{Y_n}(L^{B_k^x})-F_Y(L^{B_k^x})-\gamma_k.
    \end{align*}
    Then,
    \begin{align*}
        \sup_{x\in [0,1]^d} |F_{Y_n}(x)-F_Y(x)|
        &\leq \max_{B\in \mscr B_k^d} \max_{\twid x \in \{U^B,L^B\}} |F_{Y_n}(\twid x)-F_Y(\twid x)|+\gamma_k.
    \end{align*}
    Note that the two maximums are over a finite number of points; by pointwise convergence of $F_{Y_n}$ to $F_Y$ we have that $\max_{B\in \mscr B_k^d} \max_{\twid x \in \{U^B,L^B\}} |F_{Y_n}(\twid x)-F_Y(\twid x)|\rightarrow 0$. Combining this with the fact that $\gamma_k\rightarrow 0$, we have the result. 
\end{proof}

\begin{restatable}{lem}{lemKSmargin}\label{lem:KSmargin}
Suppose that $f_n(x|y)\overset {KS}\sim g_n(x|y)$ for almost all ($\pi$) $y$, where $\pi$ is a probability distribution on $y$. Then, 
\[f_n(x|y)\pi(y) \overset {KS}\sim g_n(x|y)\pi(y).\]
The result also holds if $KS$ is replaced with total variation distance. 
\end{restatable}
\begin{proof}
\begin{align*}
    \mathrm{KS}(f_n(x|y)\pi(y), g_n(x|y) \pi(y))
    &= \sup_{S}\sup_{R} \left| \int_S\int_R f_n(x|y)\pi(y)-g_n(x|y)\pi(y) \ dxdy\right|\\
    &\leq \sup_{S} \int \pi(y) \sup_{R} \left| \int_R f_n(x|y) -g_n(x|y)\ dx\right| \ dy\\
    &=\sup_{S} \int_S \pi(y) \mathrm{KS}(f_n(x|y),g_n(x|y))  \ dy\\
    &\leq \int \mathrm{KS}(f_n(x|y),g_n(x|y))\pi(y) \ dy\\
    &\rightarrow 0,
\end{align*}
where $R$ and $S$ are rectangles, and in the last step, we used the fact that the KS distance is non-negative and bounded above by 1 and applied the dominated convergence theorem. 

In the case of total variation distance, the proof is the same, except that $R$ and $S$ are arbitrary measurable sets.
\end{proof}

\begin{restatable}{lem}{lemminus}\label{lem:minus}
Under assumption A1 and  if $m-n=O(a_n)$ where $a_n$ is defined in A2, then as $n\rightarrow \infty$,  
\begin{enumerate}
    \item $p(n^b(n^{-a}s-g(\theta))|\theta,n)\overset {KS}\sim p(m^b(m^{-a}s-g(\theta))|\theta,n)$, and 
    \item $p(s|\theta,n)\overset {KS}\sim p(s|\theta,m)$.
\end{enumerate}
If A1' holds, then convergence for both 1. and 2. holds in total variation as well.
\end{restatable}
\begin{proof}
We will view $m$ as a function of $n$. Then for part 2, we can write
\[\mathrm{KS}(p(s|\theta,n),p(s|\theta,m))=\mathrm{KS}(p(m^b(m^{-a}s-g(\theta))|\theta,n),p(m^b(m^{-a}s-g(\theta))|\theta,m)).\]
We know that  $n^b(n^{-a}s-g(\theta))|\theta,n))\overset d \rightarrow X$ under $p(s|\theta,n)$  and $m^b(m^{-a}s-g(\theta))|\theta,m))\overset d \rightarrow X$ under $p(s|\theta,m)$. Thus,  for either part 1 or 2, it suffices to show that $m^b(m^{-a}s-g(\theta))\overset d \rightarrow X$ under $p(s|\theta,n)$,  since $X$ is continuous (Lemma \ref{lem:polya}). 

First, in the case that $a=0$, we have that 
\[m^b(s-g(\theta))=(m/n)^b n^b(s-g(\theta))\rightarrow X,\]
so long as $m-n=o(n)$. 

Next, if $a\neq 0$, then 
\begin{align*}
    m^b(m^{-a}s-g(\theta))&=(m/n)^b n^b\{(m/n)^{-a} n^{-a} s-g(\theta)\}\\
    &=(m/n)^{b-a}n^b\{n^{-a} s-(m/n)^a g(\theta)\}\\
        &=(m/n)^{b-a}n^b\{n^{-a} s-g(\theta)+g(\theta)-(m/n)^a g(\theta)\}\\
            &=(m/n)^{b-a}n^b\{n^{-a} s-g(\theta)\}+ (m/n)^{b-a}n^b\{g(\theta)-(m/n)^a g(\theta)\}.
\end{align*}
At this point, we see that the first term converges in distribution to $X$ as long as $m-n=o(n)$. It remains to show that the second term converges to 0. Recall from A2 that $a_n=o(n^{1-b})$, so $m=n+o(n^{1-b})$ or equivalently, $m/n=1+o(n^{-b})$. Substituting this into the second term above gives
\begin{align*}
    (m/n)^{b-a}n^b\{g(\theta)-(m/n)^a g(\theta)\}&=\{1+o(n^{-b})\}^{b-a} n^b [1-\{1+o(n^{-b})\}^{a}]\\
    &= \{1+(b-a)o(n^{-b})\}n^b[1-\{1+(a)o(n^{-b})\}]\\
    &=\{1+o(n^{-b})\}n^b\{o(n^{-b})\}\\
    &=o(1),
\end{align*}
where we used the series expansion of $(1+x)^p=1+px+O(x^2)$ at $x=0$ to establish the second line. In conclusion, we have shown that $m^b(m^{-a}s-g(\theta))\overset d\rightarrow X$ under $p(s|\theta,n)$, which by Lemma \ref{lem:polya} implies that  both properties 1 and 2 hold.

If all instances of convergence in distribution and convergence in KS distance are replaced by convergence in total variation, then the same proof works, leveraging Lemma \ref{lem:parthasarathy}.
\end{proof}

\begin{proof}[Proof of Theorem \ref{thm:convergence}.]
The proofs of both results are very similar. For completeness, we write out both. 

1. By Assumption A2(a), given $\gamma_1>0$, there exists $M>0$ and $N_1>0$ such that for all $n\geq N_1$, 
\[P(|\ndp-n|>M a_{n} | n)<\gamma_1.\]
Call $B_{n,M,\gamma_1} = \left\{m\in \ZZ \Big|  |m-n|\leq M a_{n}\right\}$, which has probability $\geq (1-\gamma_1)$ in $p(\ndp|n)$. Note that for all $m\in B_{n,M,\gamma_1}$, we have that $m- n|n=O(a_{n})$.

Let $\gamma_2>0$. Then from Lemma~\ref{lem:minus} since $p(n^{b}(n^{-a}s-g(\theta))|\theta,n)\overset {KS}\sim p(m^{b}(m^{-a}s-g(\theta))|\theta,n)$ when $m - n|n=O(a_{n})$, there exists $N_2>0$ such that for all $n\geq N_2$ and $m\in B_{n,M,\gamma_1}$, $\mathrm{KS}(p(n^{b}(n^{-a}s-g(\theta))|\theta,n = n_0), p(m^{b}(m^{-a}s-g(\theta))|\theta,n))<\gamma_2$ (note that we use the fact that $B_{n,M,\gamma_1}$ contains a finite number of elements).

Taking $N=\max\{N_1,N_2\}$, we have that 
\begin{align*}
   & \mathrm{KS}((p(n^{b}(n^{-a}s-g(\theta))|\theta,n), p(\ndp^b(\ndp^{-a}s-g(\theta))|\theta,n))\\
&=\sup_R \left| \int_R \sum_{m\in \ZZ} p(\ndp=m|n) p(m^b(m^{-a}s-g(\theta))|\theta,n) - p(n^b(n^{-a}s-g(\theta))|\theta,n) \ ds\right|\\
&=\sup_R\left |  \int_R \sum_{m\in \ZZ} p(\ndp=m|n) [p(m^b(m^{-a}s-g(\theta))|\theta,n) - p(n^b(n^{-a}s-g(\theta))|\theta,n)] \ ds\right|\\
&\leq  \sum_{m\in \ZZ} p(\ndp=m|n) \sup_R\left| \int_R p(m^b(m^{-a}s-g(\theta))|\theta,n) - p(n^b(n^{-a}s-g(\theta))|\theta,n) \ ds\right|\\
&=\sum_{m\not\in B_{n,M,\gamma_1}} p(\ndp=m|n)\mathrm{KS}( p(m^b(m^{-a}s-g(\theta))|\theta,n),p(n^b(n^{-a}s-g(\theta))|\theta,n))\\
&\phantom{=}+\sum_{m\in B_{n,M,\gamma_1}} p(\ndp=m|n)\mathrm{KS}( p(m^b(m^{-a}s-g(\theta))|\theta,n),p(n^b(n^{-a}s-g(\theta))|\theta,n))\\
& \leq \sum_{m\not\in B_{n,M,\gamma_1}} p(\ndp=m|n)
+\sum_{m\in B_{n,M,\gamma_1}} p(\ndp=m|n)\gamma_2\\
&\leq \gamma_1+\gamma_2, 
\end{align*}
where the supremum is over all axis-aligned rectangles R.
This establishes that $\mathrm{KS}(p(n^b(n^{-a}s-g(\theta))|\theta,n), p(\ndp^b(\ndp^{-a}s-g(\theta))|\theta,n))\rightarrow 0$, as $n\rightarrow \infty$, since the choices of $\gamma_1$ and $\gamma_2$ were arbitrary. The same argument works to show that $\mathrm{KS}(p(\theta,n^b(n^{-a}s-g(\theta))|n, p(\theta,\ndp^b(\ndp^{-a}s-g(\theta))|n))\rightarrow 0$, but we will need to use the fact that $p(\theta)$ integrates to 1.

2. By Assumption A2(b), given $\gamma_1>0$, there exists $M>0$ and $N_1>0$ such that for all $n_0\geq N_1$, 
\[P(|n-n_0|>M a_{n_0} | \ndp=n_0)<\gamma_1.\]
Call $B_{n_0,M,\gamma_1} = \left\{n\Big|  |n-n_0|\leq M a_{n_0}\right\}$, which has probability $\geq (1-\gamma_1)$ in $p(n|\ndp=n_0)$. Note that for all $m\in B_{n_0,M,\gamma_1}$, we have that $m- n_0=O(a_{n_0})$.

Let $\gamma_2>0$. Then from Lemma~\ref{lem:minus} since $p(s|\theta,n = n_0)\overset {KS}\sim p(s|\theta,n = m)$ when $m - n_0=O(a_{n_0})$, there exists $N_2>0$ such that for all $n_0\geq N_2$ and $m\in B_{n_0,M,\gamma_1}$, $\mathrm{KS}(p(s|\theta,n = n_0), p(s|\theta,n = m))<\gamma_2$ (note that we use the fact that $B_{n_0,M,\gamma_1}$ contains a finite number of elements).

Taking $N=\max\{N_1,N_2\}$, we have that 
\begin{align*}
   & \mathrm{KS}((p(s|\theta,n=n_0), p(s|\theta,\ndp=n_0))\\
&=\sup_R \left| \int_R \sum_{k=1}^\infty p(n=k|\ndp=n_0) p(s|\theta,n=k) - p(s|\theta,n=n_0) \ ds\right|\\
&=\sup_R\left |  \int_R \sum_{k=1}^\infty p(n=k|\ndp=n_0) [p(s|\theta,n=k) - p(s|\theta,n=n_0)] \ ds\right|\\
&\leq  \sum_{k=1}^\infty p(n=k|\ndp=n_0) \sup_R\left| \int_R p(s|\theta,n=k) - p(s|\theta,n=n_0) \ ds\right|\\
&=\sum_{k\not\in B_{n_0,M,\gamma_1}} p(n=k|\ndp=n_0)\mathrm{KS}( p(s|\theta,n=k),p(s|\theta,n=n_0))\\
&\phantom{=}+\sum_{k\in B_{n_0,M,\gamma_1}} p(n=k|\ndp=n_0)\mathrm{KS}( p(s|\theta,n=k),p(s|\theta,n=n_0))\\
& \leq \sum_{k\not\in B_{n_0,M,\gamma_1}} p(n=k|\ndp=n_0)
+\sum_{k\in B_{n_0,M,\gamma_1}} p(n=k|\ndp = n_0)\gamma_2\\
&\leq \gamma_1+\gamma_2, 
\end{align*}
where the supremum is over all axis-aligned rectangles R.
This establishes that $\mathrm{KS}(p(s|\theta,n=n_0), p(s|\theta,\ndp=n_0))\rightarrow 0$, as $n_0\rightarrow \infty$, since the choices of $\gamma_1$ and $\gamma_2$ were arbitrary. The same argument works to show that $\mathrm{KS}(p(s,\theta|n=n_0), p(s,\theta|\ndp=n_0))\rightarrow 0$, but we will need to use the fact that $p(\theta)$ integrates to 1.

If A1' holds, then replacing all instances of KS with TV, and having $\sup_R$ over all measurable sets rather than only rectangles, we get the results in total variation distance. 
\end{proof}

\begin{cor}\label{cor:TV}
    If A1' and A2 hold, and  $g$ is any function, then, $p(g(s)|\theta,n=n_0)\overset {TV}\sim p(g(s)|\theta,\ndp=n_0)$.
\end{cor}

The following example gives detailed derivations for Example \ref{ex:LaplaceKNG}.
\begin{example}
    \label{ex:LaplaceKNGDetails}
    First, we establish some additional notation. Let $\theta = \EE x_i$ and $\sigma^2 = \var(x_i)$. We know by the CLT that $n^{1/2}(\overline x -\theta)\overset d \rightarrow N(0,\sigma^2)$. We will re-scale our DP estimators so that they approximate $\ol x$, and thus we will have $b=1/2$ in both cases. 

    In both cases, our proof strategy will be as follows: Write $n^{1/2}(n^{-a} s-\theta) = n^{1/2}(\ol x -\theta)+n^{1/2}(n^{-a}s-\ol x)$. By Slutsky's Theorem and Continuous Mapping Theorem, it suffices to show that $n^{1/2}(n^{-a} s-\ol x)=o_p(1)$ in order to have that $n^{1/2}(n^{-a} s-\theta) \overset d \rightarrow N(0,\sigma^2)$. In fact, in both cases we will prove a stronger result: $n(n^{-a}s-\ol x)=O_p(1)$.
    
    {\bf Laplace Mechanism: }We have $n(n^{-1}s-\ol x) = s-\sum_{i=1}^n x_i\sim \mathrm{Laplace}(0,1/\ep)$, which implies that $n(n^{-1} s-\ol x)=O_p(1)$.

    {\bf KNG: } We will define three random variables for notational convenience: given $\ol x$, define $L_{\ol x}|\ol x \sim \mathrm{Laplace}_{[-n\ol x,n(1-\ol x)]}(0,2/\ep)$, which is the conditional distribution of $n(s-\ol x)|\ol x$, we also define $L\sim \mathrm{Laplace}(0,2/\ep)$ and $E\sim \mathrm{Exp}(\ep/2)$, using the rate-parametrization. Note that $|L|\overset d = E$. For a random variable $X$, we will write $F_X$ for its cdf.

    To show $n(s-\ol x)= O_p(1)$ (marginally over both the randomness in $\ol x$ and in $s$), let $\gamma>0$ be given. Set $M = F^{-1}_E(1-\gamma[F_L(1)-1/2])$. Note that $M$ satisfies $P(|L|>M)=\gamma[F_L(1)-1/2]$. Consider the following:
    \begin{align}
        P(|n(s-\ol x)|>M)&= \EE_{\ol x} \left [P(|n(s-\ol x)|>M | \ol x)\right]\\
        &=\EE_{\ol x} \left[P(|L_{\ol x}|>M | \ol x)\right]\\
        &=\EE_{\ol x}\left[\frac{P(L\in [-M,M]^c\cap [-n\ol x,n(1-\ol x)]|\ol x)}{P(L\in [-n\ol x,n(1-\ol x)] | \ol x)}\right]\\
        &\leq \EE_{\ol x}\left[ \frac{P(L\in [-M,M]^c)}{P(L\in [-n\ol x,n(1-\ol x)] | \ol x)}\right]\\
        &=\EE_{\ol x}\left[\frac{\gamma[F_L(1)-1/2]}{P(L\in [-n\ol x,n(1-\ol x)] | \ol x)}\right]\\
        &\leq \frac{\gamma[F_L(1)-1/2]}{P_L(L\in [0,1])}\label{eq:01}\\
        &=\gamma,
    \end{align}
    where in \eqref{eq:01}, we use the observation that either $[-1,0]$ or $[0,1]$ is a subset of $[-n\ol x,n(1-\ol x)]$, since $n\geq 1$ and $x_1\in \{0,1\}$. By symmetry of the Laplace distribution, we have that $P(L\in [-1,0])=P(L\in[0,1])=F_L(1)-F_L(0)=F_L(1)-1/2$. We conclude that $n(s-\ol x)=O_p(1)$, which completes the derivation.
    
\end{example}

\begin{proof}[Proof of Theorem \ref{thm:posteriorRectangle}.]
 For two vectors $x,y\in \RR^d$, we write $x\leq y$ to indicate that $x_i\leq y_i$ for all $i=1,\ldots, d$. We will also write $\int_{x}^y g(t)\ dt$ as shorthand for $\int_{x_1}^{y_1}\cdots \int_{x_d}^{y_d} g(t_1,\ldots,t_d) \ dt_1\cdots dt_{d}$.
 
Call $F_{n_0}(t)=P(n_0^{b'}(\theta-\theta_{n_0})\leq t|s\in R_{n_0},n=n_0)$, which we know satisfies $F_{n_0}(t)\rightarrow F_Y(t)$ as $n_0\rightarrow \infty$. 
Then, we consider the cdf of $n_0^{b'}(\theta-\theta_{n_0})| (s\in R_{n_0},\ndp=n_0)$, and aim to show that it converges to $F_Y(t)$:
\begin{align*}
    P(n_0^b(\theta-\theta_{n_0})\leq t | s\in R_{n_0},\ndp=n_0)
    &=\frac{P(n_0^{b'}(\theta-\theta_{n_0})\leq t, s\in R_{n_0}|\ndp=n_0)}{P(s\in R_{n_0}|\ndp=n_0)}\\
    &=\frac{\int_{-\infty}^t \int_{R_{n_0}} p(n_0^{b'}(\theta-\theta_{n_0})=x,s | \ndp=n_0)\ dsdx}
    {\int_{-\infty}^\infty \int_{R_{n_0}} p(n_0^{b'}(\theta-\theta_{n_0})=x,s | \ndp=n_0)\ dsdx}\\
    &\leq \frac{P(n_0^{b'}(\theta-\theta_{n_0})\leq t, s\in R_{n_0}|n=n_0)+\mathrm{KS}_{n_0}}
    {P(s\in R_{n_0}|n=n_0)-\mathrm{KS}_{n_0}}\\
    &=\frac{F_{n_0}(t) p_{n_0}+\mathrm{KS}_{n_0}}
    {p_{n_0}-\mathrm{KS}_{n_0}}\\
    &=\frac{F_{n_0}(t) +\mathrm{KS}_{n_0}/p_{n_0}}
    {1-\mathrm{KS}_{n_0}/p_{n_0}}\\
    &\rightarrow F_Y(t),
\end{align*}
since $F_{n_0}(t)\rightarrow F_Y(t)$, and $\mathrm{KS}_{n_0}/p_{n_0}\rightarrow 0$. 
\end{proof}

\begin{proof}[Proof of Proposition \ref{prop:TVprivacy}]
    By the data processing inequality, we can
    remove $s$ from both quantities. Note that the distributions of $n$ and $\ndp$ do not depend on $\theta$. Applying the maximal coupling inequality gives the upper bound. When $\ndp\in \ZZ$, we can apply Markov's inequality to obtain $P(\ndp\neq n)=P(|\ndp-n|\geq 1)\leq \EE_{n,\ndp}|\ndp-n|.$
\end{proof}

Nearly all mainline definitions of differential privacy imply $(0,\delta)$-DP. Some conversions are provided in Lemma \ref{lem:TVDP}. The conversions use standard techniques from the analysis of $f$-DP or inequalities relating divergences. Note that properties 1-4 are tight, whereas 5 and 6 are likely loose.

\begin{restatable}{lem}{lemTVDP}\label{lem:TVDP}
Let $M:\mscr X\rightarrow \mscr Y$ be a privacy mechanism.
\begin{enumerate}
    \item If $M$ satisfies $f$-DP  \citep{dong2021gaussian}, then it satisfies $(0,1-2c)$-DP, where $c$ is the unique fixed point of $f$: $f(c)=c$.
    \item If $M$ satisfies $\ep$-DP, then it satisfies $\left(0,\frac{\exp(\ep)-1}{\exp(\ep)+1}\right)$-DP.
    \item If $M$ satisfies $(\ep,\de)$-DP \citep{Dwork2014}, then it satisfies $\left(0,\frac{2\delta+\exp(\ep)-1}{\exp(\ep)+1}\right)$-DP.
    \item If $M$ satisfies $\mu$-GDP  \citep{dong2021gaussian}, then it satisfies $(0,2\Phi(\mu/2)-1)$-DP.
    \item If $M$ satisfies $\rho$-zCDP \citep{bun2016concentrated}, then it satisfies\\ $\left(0,\min\left\{\sqrt{\rho/2},\sqrt{1-\exp(-\rho)}\right\}\right)$-DP.
    \item If $M$ satisfies $(\alpha,\ep)$-RDP  \citep{mironov2017renyi}, then it satisfies\\ $\left(0,\min\left\{\sqrt{\rho/2},\sqrt{1-\exp(-\rho)}\right\}\right)$-DP.
\end{enumerate}
\end{restatable}
\begin{proof}
The first four properties are proved using $f$-DP. If $f$ has fixed point $c$, then $f_{0,1-2c}\leq f$  \citep{dong2021gaussian,awan2024optimizing}. For properties 2-4, we will simply calculate $c$ for the tradeoff functions $f_{\ep,0}$, $f_{\ep,\de}$ and $G_\mu$, and then apply property 1.  Properties 5 and 6 use standard inequalities relating R\'enyi divergences to KL-divergence, and then to total variation distance.
\begin{enumerate}
\item It is easy to see that if $f$ has fixed point $c$, then $f_{0,1-2c}\leq f$, since this is the tangent line of the convex function $f$ at the point $(c,c)$.
    \item The fixed point of $f_{\ep,0}$ is the solution to the following equation: $1-\exp(\ep)c=\exp(-\ep)(1-c)$, which yields $c=1/(1+\exp(\ep))$. 
    \item The fixed point of $f_{\ep,\de}$ is the solution to the following equation: $1-\delta-\exp(\ep)c=\exp(-\ep)(1-\delta-c)$, which yields $c=(1-\delta)/(1+\exp(\ep))$.
    \item The fixed point $c$ of $G_\mu(\alpha) = \Phi(\Phi^{-1}(1-\alpha)-\mu)$ satisfies $\frac{d}{d\alpha} G_\mu(\alpha)\Big|_{\alpha=c} = -1$. We calculate 
    \[\frac{d}{d\alpha} G_\mu(\alpha) =\frac{(-1) \phi(\Phi^{-1}(1-\alpha)-\mu)}{\phi(\Phi^{-1}(1-\alpha))}.\]
    We can easily verify that $1-\Phi(\mu/2)$ is the solution to $\frac{d}{d\alpha} G_\mu(\alpha) = -1$.
    \item Recall that $\rho$-zCDP is equivalent to bounding the order $\alpha$ R\'enyi Divergence by $\rho\alpha$ for all $\alpha>1$. Since R\'enyi divergences are monotonic: $D_{\alpha_1}\leq D_{\alpha_2}$ when $1\leq \alpha_1\leq \alpha_2$, and $D_1$ is the $\mathrm{KL}$-divergence, we have that the $\mathrm{KL}$-divergence is bounded above by $\rho$. Pinsker's inequality implies that $\mathrm{TV}\leq \sqrt{\rho/2}$, and Bretagnole \& Huber's inequality implies that $\mathrm{TV}\leq \sqrt{1-\exp(-\rho)}$.
    \item Recall that $(\alpha,\epsilon)$-RDP means that $D_\alpha\leq \epsilon$. Again, by the monotonicity of R\'enyi divergences, we have $\mathrm{KL}\leq \epsilon$ and apply the same inequalities as in part 5.
\end{enumerate}
\end{proof}

\begin{restatable}{lem}{lemDPminus}\label{lem:DPminus}
Let $p(s|x)$ be the distribution of a $(0,\delta)$-DP privacy mechanism, and let $p(x|\theta,n)$ be a distribution for $x$, where we assume that $x$ consists of $n$ copies of i.i.d. data. Then 
$\mathrm{TV}(p(s|\theta,n),p(s|\theta,m))\leq \delta |n-m|.$
\end{restatable}
\begin{proof}
To clarify the notation, we use $f(x|\theta,n)$ to represent the density $p(x|\theta, n)$. Then, 
\begin{align*}
\mathrm{TV}(p(s|\theta,n),p(s|\theta,m))
&=\frac 12 \int \Big|p(s|\theta,n)-p(s|\theta,m)\Big|\ ds\\
&=\frac 12 \int \Big|\int p(s|x){f}(x|\theta,n)\ dx-\int p(s|y) f(y|\theta,n)\ dy\Big|\ ds
\end{align*}
Without loss of generality, say that $m>n$. Then we can couple $x$ and $y$ such that $x\subset y$ and $y=(x,z)$. Then,
\begin{align*}
    &\mathrm{TV}(p(s|\theta,n),p(s|\theta,m))\\
    &=\frac 12 \int \Big| \int\int p(s|x) f(x|\theta,n)f(z|\theta,m-n)-p(s|(x,z))f(x|\theta,n)f(z|\theta,m-n) \ dxdz\Big| ds\\
    &\leq \int\int\frac 12 \int \Big| p(s|x)-p(s|(x,z))\Big| \ ds f(x|\theta,n)f(z|\theta,m-n)\ dxdz\\
    &\leq \int\int \delta|n-m| f(x|\theta,n)f(z|\theta,m-n) \ dxdz\\
    &=\delta|n-m|,
\end{align*}
where we used the fact that if $M$ satisfies $(0,\delta)$-DP for groups of size 1, then $M$ satisfies $(0,\min\{1,k\delta\})$-DP for groups of size $k$ \citep[Section 4.3]{awan2022log}.
\end{proof}

\begin{proof}[Proof of Theorem \ref{thm:TVprivacy}.]
We will prove the result for $\mathrm{TV}(p(s,\theta|n=n_0),p(s,\theta|\ndp=n_0))$. The other result has a nearly identical proof, except that we do not need to integrate over $p(\theta)$. 

\begin{align*}
    &\mathrm{TV}(p(s,\theta|n=n_0),p(s,\theta|\ndp=n_0))\\
    &=\frac 12\int\int \Big| p(s|\theta,n_0)p(\theta)-\sum_{k=1}^\infty p(s|\theta,k) p(\theta)p(n=k|\ndp=n_0) \Big| \ dsd\theta\\
    &=\frac 12\int\int \Big| \sum_{k=1}^\infty (p(s|\theta,n_0)-p(s|\theta,k))p(n=k|\ndp=n_0)\Big| \ ds p(\theta)\ d\theta\\
    &\leq \int \sum_{k=1}^\infty \frac 12\int \Big| p(s|\theta,n_0)-p(s|\theta,k)\Big|\ ds \ p(n=k|\ndp=n_0) p(\theta)\ d\theta\\
    &\leq \int \sum_{k=1}^\infty \delta |k-n_0| p(n=k,\ndp=n_0) p(\theta) \ d\theta\\
    &= \delta  \sum_{k=1}^\infty |k-n_0| p(n=k|\ndp=n_0)\\
    &= \delta  \EE_{n|\ndp=n_0} |n-n_0|.
\end{align*}
\end{proof}

\begin{restatable}{lem}{lemExpectation}\label{lem:Expectation}
Suppose that $p(n)$ is a flat improper prior on $\{1,2,\ldots\}$ and that $n_0\in \NN$.
\begin{enumerate}
    \item If $p(n|\ndp=n_0)\propto \exp(-\ep|n-n_0|)$, then 
$\EE_{n|\ndp=n_0} |n-n_0| \leq \frac{2}{\exp(\ep)-1}.$
\item If $p(n|\ndp=n_0)\propto \exp(-(\ep^2/2) (n-n_0)^2)$, then 
$\EE_{n|\ndp=n_0} |n-n_0| \leq 2/\ep.$

Note that both quantities go to zero when $\ep\rightarrow \infty$.
\end{enumerate}
\end{restatable}
\begin{proof}
1. First, we suppose that $p(n|\ndp=n_0)\propto \exp(-\ep|n-n_0|)$, and calculate,
\begin{align*}
    \EE_{n|\ndp=n_0}&=\frac{\sum_{k=1}^\infty |k-n_0| \exp(-\ep|k-n_0|)}{\sum_{k=1}^\infty \exp(-\ep|k-n_0|)}.
\end{align*}
We can lower bound the denominator as follows:
\[\sum_{k=1}^\infty \exp(-\ep|k-n_0|)\geq \sum_{j=0}^\infty \exp(-\ep j)=\frac{\exp(\ep)}{\exp(\ep)-1}.\]
Then,
\begin{align}
    \EE_{n|\ndp=n_0}&\leq \sum_{k=1}^\infty |k-n_0| \exp(-\ep|k-n_0|) \left(\frac{\exp(\ep)-1}{\exp(\ep)}\right)\\
    &\leq \sum_{k=-\infty}^\infty|k-n_0| \exp(-\ep|k-n_0|) \left(\frac{\exp(\ep)-1}{\exp(\ep)}\right)\\
    &=\sum_{j=-\infty}^\infty |j| \exp(-\ep j) \left(\frac{\exp(\ep)-1}{\exp(\ep)}\right)\\
    &=(1+\exp(-\ep)) \EE_{X\sim \mathrm{DLap}(\exp(-\ep))} |X|\\
    &=(1+\exp(-\ep))\left(\frac{2\exp(-\ep)}{(1+\exp(-\ep))(1-\exp(-\ep))}\right)\label{eq:dlap}\\
    &=\frac{2}{\exp(\ep)-1},
\end{align}
where in \eqref{eq:dlap}, we use the formula for the first absolute moment of the discrete Laplace distribution \citep[Proposition 2.2]{Inusah2006}.\\
2. Next, we suppose that $p(n|\ndp=n_0)\propto \exp(-(\ep^2/2) (n-n_0)^2)$. First we will call 
\[A = \sum_{k=0}^\infty \exp(-\ep^2k^2/2),\]
and note that 
\[\sum_{k=1}^\infty \exp(-(\ep^2/2)(k-n_0)^2)\geq A,\]
and 
\[\sum_{k=-\infty}^\infty \exp(-(\ep^2/2) (k-n_0)^2))=2A-1\leq 2A.\]
Then, 
\begin{align}
    \EE_{n|\ndp=n_0} |n-n_0| &=\frac{\sum_{k=1}^\infty |k-n_0| \exp(-(\ep^2/2) (k-n_0)^2)}{\sum_{k=1}^\infty \exp(-(\ep^2/2)(k-n_0)^2)}\\
    &\leq \frac{ \sum_{k=-\infty}^\infty |k-n_0| \exp(-(\ep^2/2)(k-n_0)^2)(2A)}{A \sum_{k=-\infty}^\infty \exp(-(\ep^2/2)(k-n_0)^2)}\label{eq:moment1}\\
    &=\frac{2\sum_{k=-\infty}^\infty |k| \exp(-\ep^2 k^2/2)}{\sum_{k=-\infty}^\infty \exp(-\ep^2k^2/2)}\label{eq:moment2}\\
    &=2\EE_{X\sim N_{\ZZ}(0,1/\ep^2)} |X|\\
    &\leq 2\sqrt{\EE_{X\sim N_{\ZZ}(0,1/\ep^2)} X^2}\label{eq:moment3}\\
    &=2\sqrt{\var_{X\sim N_{\ZZ}(0,1/\ep^2)} (X)}\label{eq:moment4}\\
    &\leq 2\sqrt{1/\ep^2},\label{eq:moment5}\\
    &=2/\ep,
\end{align}
where in \eqref{eq:moment1}, we multiplied top and bottom by $\sum_{k=-\infty}^\infty \exp(-(\ep^2/2)(k-n_0)^2)$ and applied the inequalities we established earlier, in terms of $A$; in \eqref{eq:moment2}, we use the fact that $n_0\in \ZZ$ and reparametrized the sums; in \eqref{eq:moment3} we use Jensen's inequality; in \eqref{eq:moment4} we express the quantity in terms of the variance of the discrete Gaussian distribution \citep{canonne2020discrete}; for \eqref{eq:moment5} we apply \citep[Corollary 9]{canonne2020discrete}.
\end{proof}

\begin{proof}[Proof of Theorem \ref{thm:ABCprivacy}.]
Let $A\subset \Theta$ be a measurable set. Then 
\begin{align*}
    P(\theta\in A|s\in S,\ndp(\gamma)=n_0)&=\frac{P(\theta\in A,s\in S|\ndp(\gamma)=n_0)}{P(s\in S|\ndp(\gamma)=n_0)}\\
    &\leq \frac{P(\theta\in A,s\in S|n=n_0)+\delta\gamma}{P(s\in S|n=n_0)-\delta\gamma}\\
    &\rightarrow P(\theta\in A| s\in S,n=n_0),
\end{align*}
where $\delta\gamma$ is the upper bound on $\mathrm{TV}(p(\theta,s|\ndp(\gamma)=n_0),p(\theta,s|n=n_0))$, derived in Theorem \ref{thm:TVprivacy}, and the convergence happens when $\delta\gamma\rightarrow 0$. 
\end{proof}

\begin{proof}[Proof of Theorem \ref{thm:mleconsistency}.]
Our proof resembles that of the argmax continuous mapping theorem in empirical processes~\citep{sen2022gental}.
According to the Portmanteau theorem [See \citet[Theorem 1.3.4 ]{van1996weak} or \citet[Theorem 10.5]{sen2022gental}], it suffices to show that, for every closed set $F \subset \Theta$, we have 
\begin{equation}\label{eq:portmanteau}
    \limsup_{\epsilon \to \infty} \mathbb{P}_{\epsilon}\left(\hat{\theta}_{\epsilon} \in F\right) \le \mathbb{P}(\hat{\theta}_0 \in F).
\end{equation}
Let $K$ be a compact subset of $\Theta$. 
Since $L(\theta)$ is upper semi-continuous and has a unique maximizer, we must have 
$$L(\hat{\theta}(n_0,s);s) > \sup_{h \in K \cap G^{c}} L(\theta;s)$$
for every open set $G$ containing $\hat{\theta}(n_0, s)$

We first observe that 
$\{\hat{\theta}_{\epsilon} \in F\} = \{\hat{\theta}_{\epsilon} \in F \cap K\} \cup \{\hat{\theta}_{\epsilon} \in F \cap K^{c}\}$ and the first set satisfies
$$\{\hat{\theta}_{\epsilon} \in F \cap K\} \subseteq \{\sup_{\theta \in F \cap K} L_{\epsilon}(\theta) \geq \sup_{\theta \in K} L_{\epsilon}(\theta) + o_{P}(1) \}.
$$
By the uniform convergence assumption and the continuous mapping theorem, we have 
$$ \sup_{\theta \in F\cap K}L_{\epsilon}(\theta) - \sup_{\theta \in K} L_{\epsilon}(\theta) + o_{P}(1) \overset{d.}{\to} \sup_{\theta \in F\cap K}L(\theta) - \sup_{\theta \in K} L(\theta).
$$
Therefore we have
$$ \lim\sup_{\epsilon \to \infty} \mathbb{P}_{\epsilon}\left(\hat{\theta}_{\epsilon} \in F \cap K\right) \le \mathbb{P}_{\epsilon}\left( \sup_{\theta \in F \cap K} L(\theta) \geq \sup_{\theta \in K} L(\theta) \right).
$$
We can divide the event on the right side by intersecting it with the partition $\{\hat{\theta}_0 \in K\}$ and $\{\hat{\theta}_0 \not\in K\}$. 
Notice that 
$\{ \sup_{\theta \in F \cap K} f_{n_0}(\theta) \geq \sup_{\theta \in K} f_{n_0}(\theta) \} \cap \{\hat{\theta}_0 \in K\} \subseteq \{\hat{\theta}_0 \in F\}.$
So we have 
$$ \lim\sup_{\epsilon \to \infty} \mathbb{P}_{\epsilon}\left(\hat{\theta}_{\epsilon} \in F \cap K\right) \le \mathbb{P}(\hat{\theta}_0 \in F) + \mathbb{P}(\hat{\theta}_0 \not\in K).
$$
So far we have shown that 
$$ \limsup_{\epsilon \to \infty} \mathbb{P}\left(\hat{\theta}_{\epsilon} \in F\right) \leq \mathbb{P}(\hat{\theta}_0 \in F) + \mathbb{P}(\hat{\theta}_0 \not\in K) + \lim\sup_{\epsilon \to \infty} \mathbb{P}_{\epsilon}(\hat{\theta}_{\epsilon} \not\in K).$$
The tightness condition implies that the last two terms can be made arbitrarily small, and therefore establishes \eqref{eq:portmanteau}.
\end{proof}

\begin{example}
   [Satisfying the Assumptions of Theorem \ref{thm:mleconsistency}]\label{ex:mleconsistency}
   In this example, we give some concrete conditions that imply the assumptions of Theorem \ref{thm:mleconsistency}. Let $\Theta$ be a compact subset of $\RR^p$ (making bullet 3 trivial). Let $T(x)|\theta,n$ be a non-private summary that is a location family in $\theta$ and is unimodal for all $\theta$ (for example a location Gaussian distribution). Let $L$ be a continuous, log-concave noise such as Laplace or Gaussian. Then $s=T(x)+L|\theta,n$ has a continuous density, which is unimodal, and is a location-family in $\theta$. This implies that $L(\theta;s,n)$ is continuous in $\theta$ and $s$ and has a unique maximum, satisfying bullet 1. 

   Let $p_\epsilon(\ndp|n)$ be either continuous Laplace or Gaussian noise with scale $1/\epsilon$, which is the privacy mechanism for $\ndp$. Let $p(n)$ be the uniform distribution on $\{1,\ldots, N\}$ for some fixed and finite $N$. Then $p(n|\ndp)$ is a proper ``posterior'' distribution for $n|\ndp$. Then, $L_\epsilon(\theta;s,\ndp)$ is continuous in $\theta$, $s$, and $\ndp$. As $\epsilon\rightarrow \infty$, we have that $\ndp\rightarrow n_0$. Furthermore, $p_\epsilon(n=k|\ndp={\ndp}_0)\rightarrow I(k={\ndp}_0)$ as $\epsilon\rightarrow \infty$. Applying Slutsky's Theorem and the Continuous Mapping Theorem, we have that $L_\epsilon(\theta;s,\ndp)\overset d \rightarrow L(\theta;s)$ as $\epsilon\rightarrow \infty$, over the randomness in $x$, $s$, and $\ndp$. This satisfies bullet 2. Thus, all of the conditions of Theorem \ref{thm:mleconsistency} have been satisfied. 
\end{example}

\begin{proof}[Proof of Proposition \ref{prop:acceptance}.]
\begin{enumerate}
    \item If a mechanism satisfies $\ep_s$-DP in the unbounded notion, then it satisfies $2\ep_s$-DP in bounded DP. This holds by the group privacy bound for groups of size 2, since swapping a row is equivalent to adding a new row and then deleting the old row. The lower bound follows from the definition of bounded DP.
    \item We see that when $p(n)$ is a flat prior, these terms cancel out, and when $n\geq 2$, the transition probabilities for $n$ and $n^*$ also cancel. The bound on the ratio of the privacy mechanisms follows from the differential privacy guarantees. In the case that $n=1$ and $n^*=2$, we have an additional factor of $1/2$, from the ratio $q(1|2)/q(2|1)=1/2$.
    \end{enumerate}
\end{proof}


\begin{proof}[Proof of Proposition \ref{prop:ergodic}.]
Call $\pi$ the distribution of $p(\theta,x_{1:n},n|s,\ndp)$. It suffices to show that the chain is aperiodic, is $\pi$-irreducible, and satisfies detailed balance for $\pi$. For example, see \citet{roberts2004general}.

    It follows from property 3 that the support of $p(\theta|x_{1:n})\propto p(\theta)p(x_{1:n}|\theta)$ agrees with $p(\theta)$, provided that $p(x_{1:n}|\theta)>0$. This implies that from almost every $x_{1:n}$ (with respect to $p(x_{1:n}|\theta)$), $p(\theta|x_{1:n})$ can propose almost every $\theta$ (with respect to  $p(\theta)$), which is accepted with probability 1. 

Note that the chain is aperiodic, since for every set $A\subset \Theta$ with $P(\theta \in A)>0$, $p(\theta|x)$ can propose $\theta\in A$ with positive probability, and the other moves in the chain always have a possibility of remaining at the current state. From properties 2 and 3, it follows from \citet{green1995reversible} that our construction satisfies detailed balance.

To see that the chain is $\pi$-irreducible, we need to show that the chain can move from any starting state to a set $A$ in a finite number of steps with positive probability, for any $A$ such that $\pi(A)>0$. By the note above, we have that $p(\theta|x)>0$ for almost all $\theta$. Furthermore, we have that $p(x_i|\theta)>0$ for almost all $x_i$. By property 3, we can transition from $n$ to $n^*$ in $|n-n^*|$ cycles. By properties 2, 3, and 4, we have that all of the acceptance probabilities are positive with probability 1 (with respect to $\pi$). Thus, we see that we can move to $A$ with positive probability in a finite number of steps. 
\end{proof}

Proposition \ref{prop:geometric} relies on \citet[Theorem 1]{qin2025geometric}, which involves some technical language. In order to state a version of \citet[Theorem 1]{qin2025geometric} convenient to our setting, we will need to introduce some additional notation.

Consider a trans-dimensional Markov chain $X(t)=(K(t), Z(t))_{t=0}^\infty$ on $\mathsf{ X}=\bigcup_{k\in \mscr K} \{k\}\times \mathsf{Z}_k$, where $(\mathsf{Z}_k,\mscr A_k)$ is a non-empty measurable space with a probability measure $\Psi_k$ and countably generated $\sigma$-algebra $\mscr A_k$, and $\mscr K$ is a finite set. The space $\mscr X$ is equipped with the $\sigma$-algebra generated by $\{k\}\times A$ for $k\in \mscr K$ and $A\in \mscr A_k$. The target distribution is $\pi$, where $\pi(\{k\}\times A) = \frac{\Psi_k(A)}{\sum_{k'\in \mscr K}\Psi_{k'}(\mathsf{Z}_{k'})}$. Note that the marginal probability of $k$ under $\pi$ is $\frac{\Psi_k(\mathsf{Z}_k)}{\sum_{k'\in \mscr K}\Psi_{k'}(\mathsf{Z}_{k'})}$ and the conditional probability of $Z\in \mathsf{Z}_k$, given $k$, is $\Phi_k(Z)=\frac{\Psi_k(Z)}{\Psi_k(\mathsf{Z}_k)}$.

\begin{lem}[\citealp{qin2025geometric}]\label{lem:qin}
   Let $X(t)$ be a trans-dimensional Markov chain satisfying the properties stated above. Assume for $k\in \mscr K$, there exists a Markov transition kernel $P_k:\mathsf{Z}_k\times \mscr A_k\rightarrow [0,1]$ such that 
   \begin{enumerate}
       \item for any $k\in \mscr K$, if $X(t)=(k,Z)$ for some $Z\in \mathsf{Z}_k$, then the chain stays at $k$ with probability $\geq c_k>0$ and makes within model moves according to $P_k$, which satisfies $\Phi_k P_k=\Phi_k$,
       \item the chain associated with $P_k$ is $\Phi_k$-irreducible and $\Phi_k$-a.e. geometrically ergodic and reversibly with respect to $\Phi_k$,
       \item there exists $s\in \ZZ^+$ such that for all $k,k'\in\mscr K$,
       \[\int_{\mathsf{Z}_k} \Phi_k(dz) P^s\left((k,z);( \{k'\}\times \mathsf{Z}_{k'})\right)>0,\]
       i.e., if $Z\sim \Psi_k$, then the chain starting from $(k,z)$ can reach $\{k'\}\times Z_{k'}$ after $s$ iterations with positive probability,
   \end{enumerate}
   then $X(t)$ is $\pi$-a.e. geometrically ergodic. 
\end{lem}
\begin{proof}
    The result follows from Theorem 1, Lemma 2 and Lemma 3 of \citet{qin2025geometric} along with some simplifying remarks in Section 1.1.
\end{proof}
    
\begin{proof}[Proof of Proposition \ref{prop:geometric}.]
For this proof, we will consider our Markov transition kernel to consist of \emph{two} full cycles of Algorithm \ref{alg:RJMCMC}. This is more convenient to analyze, because it ensures a positive probability at remaining at the same value of $n$ from one cycle to the next. 

The proof consists of checking that the conditions of \citet[Theorem 1]{qin2025geometric} are satisfied, and using \citet[Theorem 3.4]{ju2022data} to establish that the within-model moves are geometrically ergodic. 

    \begin{itemize}
      \setlength\itemsep{0em}
        \item Property 1 of Lemma \ref{lem:qin} holds, since Algorithm \ref{alg:RJMCMC} alternates between within-model and between-model moves and there is a lower bound on the probability of remaining at the same value of $n$, by Proposition \ref{prop:acceptance}.
        \item The assumptions ensure that within-model moves are geometrically ergodic \citep[Theorem 3.4]{ju2022data}, and \citet[Theorem 3.3]{ju2022data} asserts that the other conditions in property 2 of Lemma \ref{lem:qin} hold as well.
        \item Property 3 of Lemma \ref{lem:qin} follows since $p(n)$ has finite support and Algorithm \ref{alg:RJMCMC} is irreducible by Proposition \ref{prop:ergodic}.
    \end{itemize}
\end{proof}

\section{Multinomial data with unknown \texorpdfstring{$n$}{n}  }\label{s:multi}
 One of the oldest and most common data structures containing sensitive data are contingency tables \citep{Hundepool2012,kim2025differentially}, which are commonly modeled as multinomials. To achieve DP in this setting, it is common to add independent noise to the counts in the table  \citep{karwa2015private,abowd2019census,kim2025differentially}. However, in the unbounded DP framework, we do not know the sample size $n$ which is one of the parameters in the multinomial. We can use a clever relation between the multinomial and Poisson distributions to write an equivalent model that is much easier to perform inference on when $n$ is unknown. 

We begin with two facts:
\begin{itemize}
  \setlength\itemsep{0em}
    \item Let $x_i \sim \mathrm{Pois}(\lambda_i)$ be independent for $i=1,\ldots, k$. Note that $\sum_{i=1}^k x_i \sim \mathrm{Pois}(\sum_{i=1}^k \lambda_i)$. Then, $\{x_i\} | \sum_{i=1}^k x_i=n \sim \mathrm{Multi}(N,\{\frac{\lambda_i}{\sum_{j=1}^k \lambda_j}\})$. 
    \item Let $\{y_i\}_{i=1}^k|z=n \sim \mathrm{Multi}(n,\{p_i\}_{i=1}^k)$ (that is $z=\sum_{j=1}^k y_j$), and suppose that $z\sim \mathrm{Pois}(\lambda)$. Then  $y_i \sim \mathrm{Pois}(\lambda p_i)$ and are independent. 
\end{itemize}

The two observations above indicate that if we are okay with viewing $N$ as a random variable (with distribution $\mathrm{Pois}(\lambda)$) rather than an unknown parameter, then we can model our multinomial counts as Poisson random variables. 

\begin{remark}
We argue that it makes more sense to view $N$ as an unobserved random variable, than as an unknown parameter: typically parameters are considered to be population quantities, whereas latent/missing variables are related to the sample at hand, but are simply not observed. We can see that $N$ is in fact a function of the (unobserved) sample, and so it makes more sense to treat it as a latent random variable. 
\end{remark}

We will proceed to write down two models that are equivalent.

{\bf Model 1:}

\begin{align*}
\{\alpha_i\},\theta&\sim \text{hyper priors}\\
\lambda &\sim \mathrm{Gamma}(\sum_{i=1}^k \alpha_i,\theta)\\
    \{p_i\}_{i=1}^k &\sim \mathrm{Dirichlet}(\{\alpha_i\}_{i=1}^k)\\
    n&\sim \mathrm{Pois}(\lambda)\\
    x_1,\ldots, x_k \mid s &\sim \mathrm{Multi}(n, \{p_i\})\\
    s&\sim p(s\mid x_1,\ldots, x_k).
\end{align*} 

{\bf Model 2:}

In the following, we use the change of variables: $\twid \lambda_i = \lambda p_i$ for $i=1,\ldots, k$.

\begin{align*}
\{\alpha_i\},\theta&\sim \text{hyper priors}\\
\twid \lambda_i &\sim \mathrm{Gamma}(\alpha_i,\theta), \quad \text{for} \quad i=1,\ldots, k\\
x_i&\sim \mathrm{Pois}(\twid \lambda_i), \quad \text{for} \quad i=1,\ldots, k\\
    s&\sim p(s\mid x_1,\ldots, x_k),
\end{align*}
Note that we can compute $\lambda = \sum_{i=1}^k \twid \lambda_i$, $n = \sum_{i=1}^k x_k$, and $p_i = \twid\lambda_i/\lambda$, yet none of these quantities need to be explicitly considered in the model. If we do not care about having a conjugate prior, instead of modeling $\twid \lambda_i$ directly, we still have the option of separately modeling $\lambda$ (the expected value and variance of $n$) and $\{p_i\}$. An empirical Bayes approach may estimate $\lambda$ using a privatized estimator for $n$. 

One quirk of this model is that by modeling $n\sim \mathrm{Pois}(\lambda)$, we allow for $n\geq 0$ instead of $N\geq 1$ as assumed in the rest of the paper. In practice, this difference should not cause a significant difference in the resulting inference.

One major upside of this model is that whether we are doing an augmented MCMC sampler, or another type of simulation-based inference, there is no need for the memory or run-time to explicitly depend on $n$. For example, consider the following Metropolis-within-Gibbs sampler which is inspired by the data augmentation MCMC of \citet{ju2022data}:

\begin{enumerate}
    \item Update $\theta |x$
    \item For $i=1,\ldots, k$, set $x_i' = x_i \pm 1$, with probability $p(x_i'|x_i)$ (could be w.p. 1/2 unless $x_i=0$), and accept with probability 
    \[\frac{p(s|x') p(x'_i\mid \twid\lambda_i) p(x_i'|x_i)}{p(s|x) p(x_i|\twid\lambda_i) p(x_i'|x_i)}\Big]^1,\]
    and notice that the first ratio is bounded between $e^{-\ep}$ and $e^{\ep}$ (since changing $x_i$ to $x_i'$ can be achieved by adding/deleting a person), the second is a ratio of Poisson pdfs, which is either $\twid \lambda/(x_i+1)$ if $x_i'=x_i+1$ or $x_i/\twid \lambda$ if $x'_i=x_i$, and the final ratio is bounded between $1/2$ and $2$. 
\end{enumerate}
We see that a cycle only takes $O(k)$ time instead of $O(n)$ time, and there is no need to consider $n$ in the sampler.

\begin{figure}[ht]
    \centering
    \includegraphics[width=.48\linewidth]{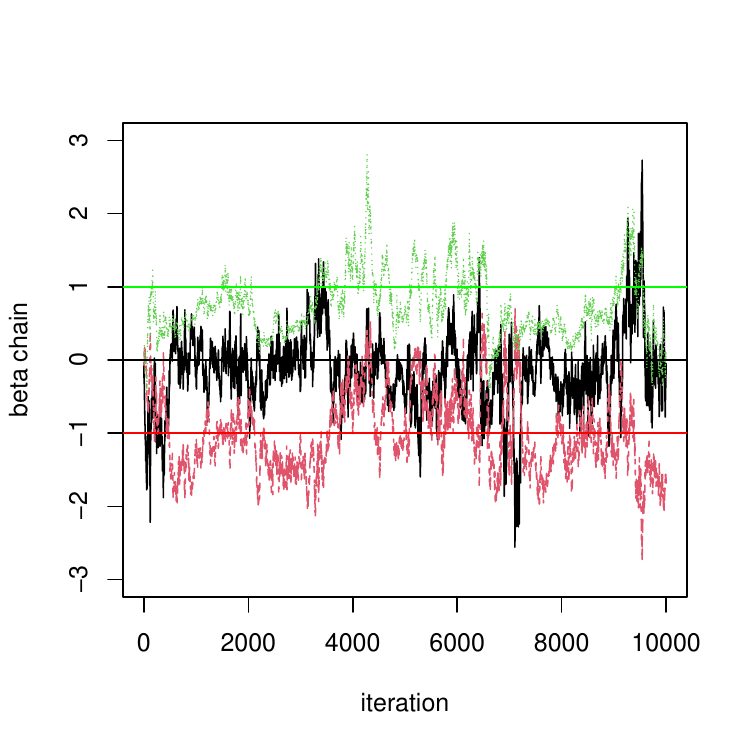}
    \includegraphics[width=.48\linewidth]{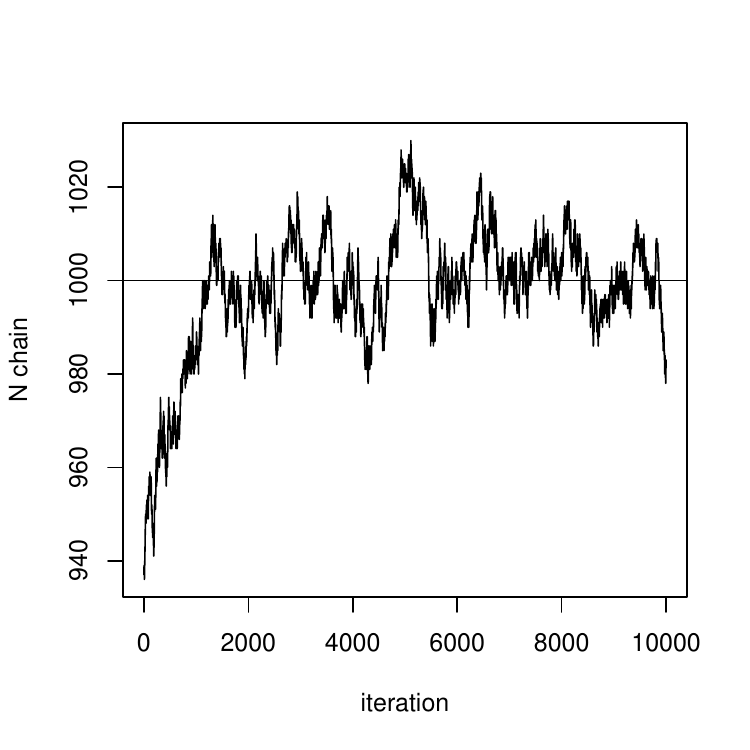}
    \caption{A plot of a single chain for the linear regression experiment, using $\ep_s=1$ and $\ep_N=.01$. 10000 iterations were completed. While a burn in of 5000 is used for the calculations, all iterations of the chain are plotted here. The initial value for $\beta$ is $(0,0,0)$ and the initial value for $N$ is $\ndp = 938$. The true values of $\beta$ are $(0,-1,1)$, illustrated by the horizontal lines.}
    \label{fig:LR}
\end{figure}

\begin{table}[t]
\centering

\begin{tabular}{lrrrrrr}
\hline
$\ep_n=$        & 0.001    & 0.01    & 0.1     & 1       & 10       & Inf      \\ \hline
$\EE(\beta_0)$  & -0.716   & -0.729  & -0.705  & -0.730  & -0.707   & -0.689   \\
$\var(\beta_0)$ & 6.098    & 1.263   & 1.178   & 1.118   & 1.184    & 1.117    \\
$\EE(\beta_1)$  & -0.568   & -0.507  & -0.540  & -0.550  & -0.457   & -0.523   \\
$\var(\beta_1)$ & 5.559    & 1.024   & 1.053   & 0.942   & 0.936    & 0.879    \\
$\EE(\beta_2)$  & 0.665    & 0.517   & 0.541   & 0.514   & 0.551    & 0.490    \\
$\var(\beta_2)$ & 3.797    & 0.925   & 0.992   & 0.884   & 0.992    & 0.787    \\
$\EE(\tau)$     & 1.052    & 1.057   & 1.025   & 1.042   & 1.026    & 1.052    \\
$\var(\tau)$    & 0.633    & 0.623   & 0.531   & 0.548   & 0.535    & 0.544    \\
$\EE(n)$        & 1116.219 & 987.717 & 998.707 & 999.874 & 1000.000 & 1000.000 \\
$\var(n)$       & 1005.443 & 762.425 & 145.776 & 1.924   & 0.005    & 0.000    \\ \hline
\end{tabular}

\vspace{.25cm}

\begin{tabular}{lrrrrrr}
\hline
$\ep_n=$        & 0.001    & 0.01     & 0.1     & 1       & 10       & Inf      \\ \hline
$\EE(\beta_0)$  & -0.282   & -0.131   & -0.115  & -0.134  & -0.134   & -0.127   \\
$\var(\beta_0)$ & 0.411    & 0.287    & 0.317   & 0.309   & 0.307    & 0.304    \\
$\EE(\beta_1)$  & -0.921   & -0.942   & -0.952  & -0.946  & -0.946   & -0.956   \\
$\var(\beta_1)$ & 0.205    & 0.190    & 0.198   & 0.196   & 0.197    & 0.195    \\
$\EE(\beta_2)$  & 0.777    & 0.880    & 0.886   & 0.875   & 0.873    & 0.867    \\
$\var(\beta_2)$ & 0.253    & 0.272    & 0.273   & 0.265   & 0.273    & 0.253    \\
$\EE(\tau)$     & 1.234    & 1.095    & 1.107   & 1.081   & 1.090    & 1.093    \\
$\var(\tau)$    & 0.351    & 0.291    & 0.296   & 0.274   & 0.286    & 0.293    \\
$\EE(n)$       & 1217.458 & 1000.015 & 999.109 & 999.871 & 1000.000 & 1000.000 \\
$\var(n)$       & 4051.367 & 258.795  & 66.134  & 1.665   & 0.005    & 0.000    \\ \hline
\end{tabular}  

\caption{Additional results for Table 1: Linear regression posterior mean and variance simulation results using $\ep_s=.1$ (top) or $\ep_s=1$ (bottom). Results are averaged over 100 replicates in each setting (one failed replicate for $\ep_n=.001$ with both $\ep_s=.1$ and $\ep_s=1$). 
Chains were run with 10,000 iterations, and used a burn-in of 5,000.
Note that $\ep_n=\mathrm{Inf}$ corresponds to bounded DP.}
\label{tab:sim01ext}
\end{table}

\begin{table}[ht] 
\centering

\begin{tabular}{lrrrrrr}
\hline
$\ep_N=$            & 0.001   & 0.01   & 0.1   & 1     & 10    & Inf   \\ \hline
$\EE(\beta_0)$\_SE  & 0.080   & 0.040  & 0.044 & 0.045 & 0.044 & 0.041 \\
$\var(\beta_0)$\_SE & 3.497   & 0.120  & 0.112 & 0.095 & 0.092 & 0.101 \\
$\EE(\beta_1)$\_SE  & 0.079   & 0.062  & 0.057 & 0.048 & 0.060 & 0.048 \\
$\var(\beta_1)$\_SE & 3.593   & 0.109  & 0.137 & 0.100 & 0.102 & 0.094 \\
$\EE(\beta_2)$\_SE  & 0.120   & 0.065  & 0.052 & 0.053 & 0.057 & 0.057 \\
$\var(\beta_2)$\_SE & 1.804   & 0.105  & 0.150 & 0.103 & 0.138 & 0.072 \\
$\EE(\tau)$\_SE     & 0.045   & 0.043  & 0.039 & 0.039 & 0.040 & 0.039 \\
$\var(\tau)$\_SE    & 0.042   & 0.048  & 0.036 & 0.042 & 0.037 & 0.032 \\
$\EE(n)$\_SE        & 101.040 & 13.060 & 1.351 & 0.140 & 0.002 & 0.000 \\
$\var(n)$\_SE       & 96.843  & 47.688 & 9.806 & 0.027 & 0.002 & 0.000 \\ \hline
\end{tabular}

\vspace{.5cm}

\begin{tabular}{lrrrrrr}
\hline
$\ep_N=$            & 0.001   & 0.01   & 0.1   & 1     & 10    & Inf   \\ \hline
$\EE(\beta_0)$\_SE  & 0.031   & 0.027  & 0.026 & 0.025 & 0.027 & 0.025 \\
$\var(\beta_0)$\_SE & 0.039   & 0.013  & 0.016 & 0.014 & 0.014 & 0.015 \\
$\EE(\beta_1)$\_SE  & 0.027   & 0.030  & 0.028 & 0.029 & 0.028 & 0.028 \\
$\var(\beta_1)$\_SE & 0.014   & 0.010  & 0.011 & 0.011 & 0.012 & 0.011 \\
$\EE(\beta_2)$\_SE  & 0.031   & 0.032  & 0.028 & 0.028 & 0.031 & 0.030 \\
$\var(\beta_2)$\_SE & 0.015   & 0.014  & 0.012 & 0.013 & 0.012 & 0.012 \\
$\EE(\tau)$\_SE     & 0.043   & 0.030  & 0.030 & 0.029 & 0.030 & 0.031 \\
$\var(\tau)$\_SE    & 0.021   & 0.017  & 0.018 & 0.015 & 0.016 & 0.017 \\
$\EE(n)$\_SE        & 76.632  & 1.328  & 0.627 & 0.129 & 0.002 & 0.000 \\
$\var(n)$\_SE       & 372.407 & 44.712 & 4.963 & 0.023 & 0.002 & 0.000 \\ \hline
\end{tabular}

\caption{Monte Carlo standard errors for the estimates in Table \ref{tab:sim01}. Top: $\ep_s=.1$, Bottom: $\ep_s=1$.}
\label{tab:sim01se}
\end{table}

\section{Additional numerical results}\label{s:extraSim}
In this section, we include additional numerical results corresponding to Section \ref{s:simulations}.

\subsection{Implementation of computational methods}
\noindent {\bf RJMCMC: }
In terms of data structure, we propose to first guess an estimate of $N>n$, such as $N=2\max\{1,\lceil \ndp\rceil\}$, and fill the first $n$ entries of the array with $x_i$ values and leave the last $N-n$ as NaNs. As we update $n$ and $x$, if we delete a row, we simply update the value of $n$ and ignore the old row; if we add a row, we place it's value in the first available spot and update $n=n+1$. If we propose $n^*>N$, then we make a new array with size $2N$ and move our database to the new array. 
By doubling the size, we reduce the frequency of creating and copying large arrays. 

\noindent{\bf MCEM: }
When using MCMC to approximate $p(x_{1:n}, n \mid \sdp,\theta^{(t+1)})$, chains can be initialized with samples from  $p(x_{1:n}, n \mid \sdp, \theta^{(t)})$, to require less burn-in.

\subsection{Linear regression}
 First, we derive the $\ell_1$ sensitivity of the unique entries in $\twid X^\top \twid X$, $\twid X^\top \twid Y$, and $\twid Y ^\top \twid Y$. The entries are of the form $\sum_i \twid X_i^{(j)}$, $\sum_i (\twid X_i^{(j)})^2$, $\sum_i \twid X_i^{(j)}\twid X_i^{(k)}$, $\sum_i \twid X_i^{(j)}\twid Y_i$, $\sum_i \twid Y_i$, and $\sum_i \twid Y_i^2$, for $i=1,\ldots, n$ and $1\leq j<k\leq p$, where $\twid X_i^{(j)}$ represents the $j^{th}$ feature of the $i^{th}$ individual/row. Since the quantities in $\twid X$ and $\twid Y$ have been normalized to take values in $[-1,1]$, we see that each entry has sensitivity 1 in unbounded DP. Thus, to calculate the $\ell_1$ sensitivity, we need to count the total number of unique entries. In the same order as above, we count 
\[\Delta = p+p+\binom{p}{2}+p+1+1=p^2/2+(5/2)p+2.\]

In Figure \ref{fig:LR} we include a visual plot of a single chain for the linear regression experiment using $\ep_s=1$ and $\ep_n=.01$. In the left portion, we plot the values of the $\beta$ vector, which we see converge to a stationary distribution centered at the true values, indicated by horizontal lines. In the right portion, we plot the values for $n$, starting at $\ndp=938$. We see that the chain quickly converges to a stationary distribution centered around the true value of 1000.

In Table \ref{tab:sim01se} we include the Monte Carlo standard errors corresponding to the results in Table \ref{tab:sim01}. We see that the standard errors for the estimates involving $\beta$ and $\tau$ are stable as $\ep_n$ varies, with the exception of $\var(\beta)$, which has an inflated standard error when $\ep_s=.1$ and $\ep_n=.001$. On the other hand, we see that the uncertainty in estimates involving $n$ does vary significantly depending on the value of $\ep_n$. These results suggest that the value of $\ep_n$ is not very important when learning about parameters other than $n$.

\begin{table}[t]

\centering

\begin{tabular}{lrrrrrr}
\hline
$\ep_n=$        & 0.001     & 0.01     & 0.1      & 1        & 10       & Inf      \\ \hline
$\EE(\beta_0)$  & -0.176    & -0.282   & -0.303   & -0.305   & -0.305   & -0.305   \\
$\var(\beta_0)$ & 1.357     & 0.715    & 1.682    & 2.322    & 2.452    & 2.468    \\
$\EE(\beta_1)$  & -0.232    & -0.259   & -0.275   & -0.280   & -0.280   & -0.280   \\
$\var(\beta_1)$ & 0.773     & 0.977    & 1.443    & 1.614    & 1.642    & 1.646    \\
$\EE(\beta_2)$  & 0.176     & 0.242    & 0.247    & 0.248    & 0.248    & 0.248    \\
$\var(\beta_2)$ & 0.905     & 1.004    & 1.486    & 1.991    & 2.001    & 2.003    \\
$\EE(\tau^{-1})$     & 1.539     & 1.539    & 1.481    & 1.456    & 1.455    & 1.455    \\
$\EE(\max\{\ndp, p+1\})$      & 4429.898 & 1046.249 & 982.472 & 998.247 & 999.825 & 1000.000 \\ 
$\EE(\ndp + r)$      & 15316.610 & 9129.091 & 9037.576 & 9041.638 & 9042.245 & 9042.319 \\ \hline
\end{tabular}

\vspace{.25cm}

\begin{tabular}{lrrrrrr}
\hline
$\ep_n=$        & 0.001     & 0.01     & 0.1      & 1        & 10       & Inf      \\ \hline
$\EE(\beta_0)$  & -0.154    & -0.306   & -0.251   & -0.110   & -0.107   & -0.107   \\
$\var(\beta_0)$ & 16.238    & 12.353   & 2.485    & 1.216    & 1.186    & 1.183    \\
$\EE(\beta_1)$  & -0.609    & -0.801   & -0.921   & -0.960   & -0.960   & -0.960   \\
$\var(\beta_1)$ & 1.644     & 2.177    & 1.246    & 0.499    & 0.487    & 0.485    \\
$\EE(\beta_2)$  & 0.508     & 0.691    & 0.825    & 0.870    & 0.871    & 0.871    \\
$\var(\beta_2)$ & 1.871     & 2.015    & 1.442    & 0.619    & 0.604    & 0.603    \\
$\EE(\tau^{-1})$     & 1.435     & 1.470    & 0.985    & 0.977    & 0.976    & 0.976    \\
$\EE(\max\{\ndp, p+1\})$      & 4429.898 & 1046.249 & 982.472 & 998.247 & 999.825 & 1000.000 \\ 
$\EE(\ndp + r)$      & 10052.274 & 1744.212 & 1152.750 & 1158.295 & 1159.423 & 1159.551 \\ \hline
\end{tabular}

\caption{Linear regression results via parametric bootstrap using $\ep_s=.1$ (top) or $\ep_s=1$ (bottom). Results are averaged over 100 replicates. The results for $\hat\beta$ for each replicate and setting are computed using 10,000 bootstrap samples. The results for $\hat\tau$ and $\ndp+r$ are based on the regularized DP version of $A^\top A$. 
Note that $\ep_n=\mathrm{Inf}$ corresponds to bounded DP. One replicate that produced extremely inflated $\var(\beta_j)$ was excluded from the results.}
\label{tab:boots}
\end{table}

To our knowledge, no existing DP methods handle unknown sample size in a principled way. For comparison, we use a version of the parametric bootstrap \citep{ferrando2022parametric}, adapted by \cite{barrientos2024feasibility}, which uses a plug-in DP estimate of $n$, needed to estimate $\tau^{-1}$. 
Specifically, the approach privatizes $A^\top A$ by adding noise $E$, where $E_{ij} = E_{ji} \sim \mathrm{Laplace}(0, b)$, with $b = 1/\epsilon_n$ if $i = j = 1$ and $b = \Delta/\epsilon_s$ otherwise. Because $A^\top A + E$ may not be positive definite (which may lead to $\hat\tau_{\rm MLE}^{-1} < 0$), a diagonal matrix $r I_{p+1}$ is added, where $r = \max(0, -2\lambda_{\min})$ and $\lambda_{\min}$ is the minimum eigenvalue of $A^\top A + E$. Since the (1,1) entry of $A^\top A$ corresponds to the sample size, they use the (1,1) entry of $A^\top A + E + r I_{p+1}$ as its privatized version, denoted $\ndp + r$, where $\ndp = n + E_{11}$.


In Table \ref{tab:boots}, when $\epsilon_s = 0.1$, $\var(\hat\beta_j)$ decreases as $\epsilon_n$ increases, due to stronger regularization (larger $r$) inflating $\ndp + r$ and artificially reducing variance. When $\epsilon_s = 1$, variance decreases as expected with increasing $\epsilon_n$ but remains higher than in our method. In general, small $\epsilon_s$ or $\epsilon_n$ inflates $\EE(\ndp + r)$ beyond the true sample size. Increasing $\epsilon_s$ reduces $r$, which lowers noise but also decreases $\ndp + r$, sometimes raising $\var(\hat\beta_j)$. Our method avoids these issues by modeling uncertainty in the sample size, rather than relying on a plug-in strategy that can be unstable under small privacy budgets.

Finally, when privately estimating $\tau^{-1}$, instead of plugging in $\ndp + r$, one could use $\max\{\ndp, p+1\}$, which may be a more natural candidate. However, using $\max\{\ndp, p+1\}$ results in estimates of $\tau^{-1}$ that are too large, ranging from 13,500 to 10 when either $\epsilon_n$ or $\epsilon_s$ is small. Since using $\max\{p+1, \ndp\}$ leads to overestimated $\tau^{-1}$, it is better to use $\ndp + r$ for the plug-in. However, as shown in Table \ref{tab:boots}, $\max\{p+1, \ndp\}$ has lower bias than $\ndp + r$, and therefore, if users need an estimate of the sample size, they should use $\max\{\ndp, p+1\}$.

Without DP, $\tau^{-1}$ can be estimated from $A^\top A$ (with $A = (\tilde X, \tilde Y)$) as $\hat\tau_{\rm MLE} = (\tilde Y^\top \tilde Y - \tilde Y^\top \tilde X (\tilde X^\top \tilde X)^{-1} \tilde X^\top \tilde Y)/(n-p)$; to obtain a private estimate of $\tau^{-1}$, \cite{barrientos2024feasibility} uses a DP version of $A^\top A$ and plugs it into the non-private estimator $\hat\tau_{\rm MLE}^{-1}$.

\subsection{Real data application}

In this section we include additional simulations corresponding to Section \ref{s:realData}. 
In Figure \ref{fig:ATUS2}, we visually compare the posterior distributions of $\alpha_1$, $\alpha_2$, and $\alpha_3$.  Figure \ref{fig:ATUS3} displays the predictive distribution for two compositional components. We observe that the predictive distributions are quite robust to the specifications of $\epsilon_s$ and $\epsilon_n$, a pattern that was also noted in \cite{guo2024differentially}.

\begin{figure}[hb!]
    \centering
    \includegraphics[width=.32\linewidth, page = 1]{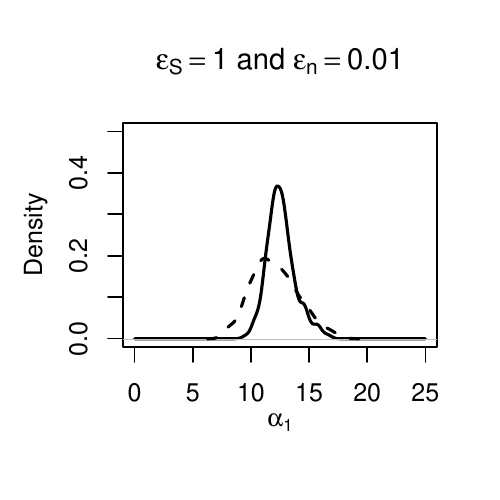}
    \includegraphics[width=.32\linewidth, page = 2]{ATUS_figures_alpha_post.pdf}
    \includegraphics[width=.32\linewidth, page = 3]{ATUS_figures_alpha_post.pdf}

\vspace{-1cm}
    
    \includegraphics[width=.32\linewidth, page = 4]{ATUS_figures_alpha_post.pdf}
    \includegraphics[width=.32\linewidth, page = 5]{ATUS_figures_alpha_post.pdf}
    \includegraphics[width=.32\linewidth, page = 6]{ATUS_figures_alpha_post.pdf}

\vspace{-.5cm}
    
    \caption{Posterior distributions for $\alpha_1$, $\alpha_2$, and $\alpha_3$ with $\epsilon_{s} \in \{1, 10\}$ and $\epsilon_{n} \in \{0.01, 0.1, 1\}$. The solid and dashes lines represent the posterior distribution under bounded and unbounded differential privacy, respectively. The displayed densities for unbounded DP correspond to the densities obtained from one of the 10 realizations of $\ndp$. The results for the remaining realizations present similar patterns.}
    \label{fig:ATUS2}
\end{figure}
\begin{figure}[ht]
    \centering
    \includegraphics[width=.48\linewidth, page = 3]{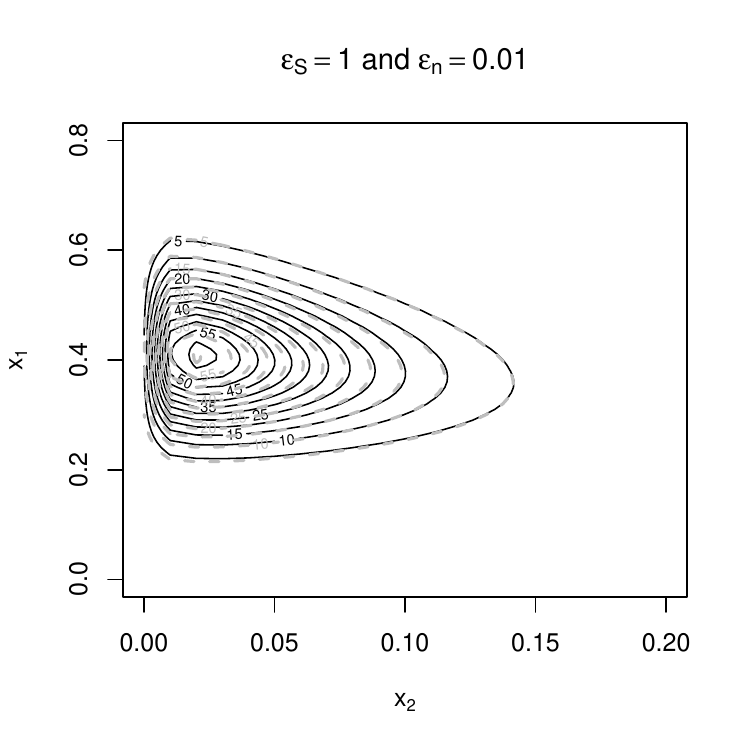}
    \includegraphics[width=.48\linewidth, page = 6]{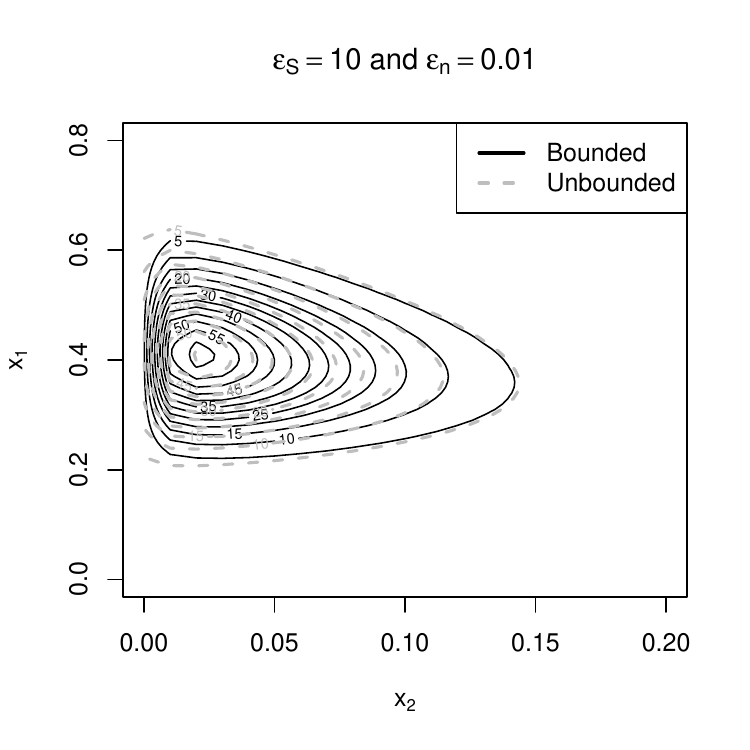}
    \caption{Predictive distributions for $(x_1, x_2)$ with $\epsilon_{s} \in \{1, 10\}$ and $\epsilon_{n} = 0.01$. The displayed densities for unbounded DP correspond to the densities obtained from one of the 10 realizations of $\ndp$. The results for the remaining realizations present similar patterns.} 
    \label{fig:ATUS3}
\end{figure}

\end{document}